\documentclass[twoside,leqno,12pt]{article}

\usepackage{amssymb}
\usepackage{amsmath}
\usepackage{amsthm}
\usepackage{ifthen}
\usepackage{calc}
\usepackage{graphicx}
\usepackage{color}
\usepackage{multirow}

\newenvironment{blue}{\color{blue}}{}

\definecolor{darkgreen}{rgb}{0,0.5,0}

\numberwithin{equation}{section}
\newtheorem{thm}[equation]{\sc Theorem}
\newtheorem{lem}[equation]{\sc Lemma}

\newtheorem{prop}[equation]{\sc Proposition}

\newtheoremstyle{notation}{3pt}{3pt}{}{}{\itshape}{:}{.5em}{\thmname{#1}}
\theoremstyle{notation}

\newtheorem{rem}{\it Remark}

\newtheorem{defin}{\it Definition}
\newtheorem{ex}{\it Example}

\makeatletter
\renewcommand{\@seccntformat }[1]{\csname the#1\endcsname. }
\makeatother

\chardef\xnearrow='045
\chardef\xnwarrow='055
\chardef\xsearrow='046
\chardef\xswarrow='056

 \newcommand\ind{\mbox{{\rm ind}}}

 \newcommand\Hom{\mbox{\rm Hom}}
 
 \newcommand\Aut{\mbox{\rm Aut}}
 
 \newcommand\Coker{\mbox{\rm Coker}}

\newcommand\arcleq{\leq_{\sf arc}}

\newcommand\extleq{\leq_{\sf ext}}

\newcommand\rep{\mbox{{\rm rep}}}
\newcommand\mrep{\mbox{ $_2{\mathbb V}_\gamma^\beta$}}

\input prepictex   \input pictex
\input postpictex

\newcounter{boxsize}
\newcounter{tempcounter}
 \newcommand\arr[2]{\arrow <2mm> [0.25,0.75] from #1 to #2}

\newcommand\ssize{\scriptscriptstyle}
\newcommand\boxes[2]{\ifthenelse{#2=3}{$\scriptstyle P_2^{#1}$}{%
                                       $\scriptstyle P_{#2}^{#1}$}}

\begin{document}

\thispagestyle{empty}
\color{black}
\phantom m\vspace{-2cm}
%\noindent{\footnotesize[{\tt \ver,} \today]}

\bigskip\bigskip
\begin{center}
{\large\bf Operations on Arc Diagrams and
Degenerations for Invariant Subspaces of Linear Operators. Part II} 
\end{center}

\smallskip

\begin{center}
Mariusz Kaniecki, Justyna Kosakowska and Markus Schmidmeier 
\footnote{The first two named authors are partially supported 
by Research Grant No.\ DEC-2011/ 02/A/ ST1/00216
of the Polish National Science Center.}
\footnote{This research is partially supported 
by a~Travel and Collaboration Grant from the Simons Foundation (Grant number
245848 to the third named author).}

\vspace{1cm}

\medskip \parbox{10cm}{\footnotesize{\bf MSC 2010:} 
Primary: 14L30, % (group actions on varieties),
16G20, % (representations of quivers),
Secondary:
16G70, % (Auslander-Reiten sequences),
05C85,  % (graph algorithms),
47A15  % (invariant subspaces)
}

\medskip \parbox{10cm}{\footnotesize{\bf Key words:} 
degenerations, partial orders, Hall polynomials,
nilpotent operators, invariant subspaces, Littlewood-Richardson
tableaux}

\end{center}

%\thanks{The first named author is partially supported
%by Research Grant No.\ DEC-2011/02/A/ST1/00216
%of the Polish National Science Center}

{\it Dedicated to Professor Daniel Simson on the occasion of his 75th birthday}

\begin{abstract}
For a partition $\beta$, denote by $N_\beta$ the nilpotent linear operator
of Jordan type $\beta$.  Given partitions $\beta$, $\gamma$, we investigate
the representation space ${}_2{\mathbb V}_{\gamma}^\beta$ of all short exact sequences
$$ \mathcal E: 0\to N_\alpha \to N_\beta \to N_\gamma \to 0$$
where $\alpha$ is any partition with each part at most 2.
Due to the condition on $\alpha$, the isomorphism type of a sequence
$\mathcal E$ is given by an arc diagram $\Delta$; denote by ${\mathbb V}_\Delta$
the subset of ${}_2{\mathbb V}_{\gamma}^\beta$ of all sequences isomorphic to $\mathcal E$.

Thus, the space ${}_2{\mathbb V}_{\gamma}^\beta$ carries a stratification given by the
subsets of type ${\mathbb V}_\Delta$. 
We compute the dimension of each stratum and 
show that the boundary of a stratum ${\mathbb V}_\Delta$ consists exactly of those
${\mathbb V}_{\Delta'}$ where $\Delta'$ is obtained from $\Delta$ by a non-empty
sequence of arc moves of five possible types {\bf (A) -- (E)}.

The case where all three partitions are fixed has been studied in [3] and [4].
There, arc moves of types {\bf (A) -- (D)} suffice to describe the boundary
of a ${\mathbb V}_\Delta$ in ${\mathbb V}_{\alpha,\gamma}^\beta$.  Our fifth move {\bf (E)}, ``explosion'', is needed to break up an arc into two poles to allow for changes
in the partition $\alpha$.
\end{abstract}

%\end{center}

%\maketitle

\section{Introduction}
%=====================

In this paper we generalize results presented in \cite{kos-sch} and \cite{ks-survey} where
arc diagrams were applied to investigate the degeneration order for  a
certain class of invariant subspaces of nilpotent linear operators. We enlarge the set of arc-moves
introduced in \cite{kos-sch} and compare the partial order induced by these moves
with the partial order given by degenerations in a~variety associated with
invariant subspaces of nilpotent linear operators.

\subsection{Arc diagrams and the arc-order}
%------------------------------------

Two diagrams of arcs and poles are said to be in {\it arc-order}
if the first is obtained from the second by a sequence of moves of type
$$\setlength{\unitlength}{0.09cm}
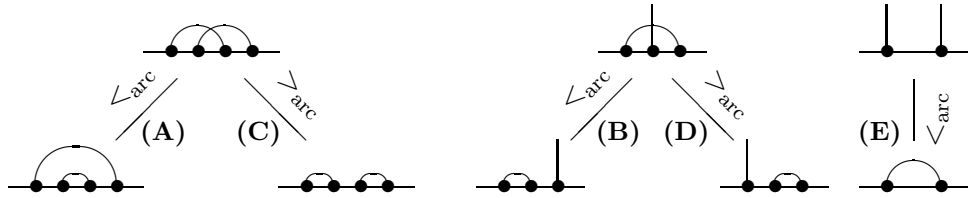
\begin{picture}(60,28)(0,-5)
  \put(0,0){\begin{picture}(20,8)(0,5)
      \put(0,4){\line(1,0){20}}
      \multiput(3,3)(4,0)4{$\bullet$}
      \put(10,4){\oval(4,4)[t]}
      \put(10,4){\oval(12,12)[t]}
    \end{picture}
  }
  \put(20,20){\begin{picture}(20,8)(0,5)
      \put(0,4){\line(1,0){20}}
      \multiput(3,3)(4,0)4{$\bullet$}
      \put(8,4){\oval(8,8)[t]}
      \put(12,4){\oval(8,8)[t]}
    \end{picture}
  }
  \put(25,15){\line(-1,-1){9}}
  \put(23,7){\makebox(0,0){\bf\footnotesize (A)}}
  \put(17,13){\rotatebox{45}{\makebox[0mm]{$<_{\rm arc}$}}}
  \put(35,15){\line(1,-1){9}}
  \put(37,7){\makebox(0,0){\bf\footnotesize(C)}}
  \put(42,12){\rotatebox{-45}{\makebox[0mm]{$>_{\rm arc}$}}}
  \put(40,0){\begin{picture}(20,8)(0,5)
      \put(0,4){\line(1,0){20}}
      \multiput(3,3)(4,0)4{$\bullet$}
      \put(6,4){\oval(4,4)[t]}
      \put(14,4){\oval(4,4)[t]}
    \end{picture}
  }
\end{picture}
\qquad
\begin{picture}(52,28)(0,-5)
  \put(0,0){\begin{picture}(16,8)(0,5)
      \put(0,4){\line(1,0){16}}
      \multiput(3,3)(4,0)3{$\bullet$}
      \put(6,4){\oval(4,4)[t]}
      \put(12,4){\line(0,1){7}}
    \end{picture}
  }
  \put(18,20){\begin{picture}(16,8)(0,5)
      \put(0,4){\line(1,0){16}}
      \multiput(3,3)(4,0)3{$\bullet$}
      \put(8,4){\oval(8,8)[t]}
      \put(8,4){\line(0,1){7}}
    \end{picture}
  }
  \put(36,0){\begin{picture}(16,8)(0,5)
      \put(0,4){\line(1,0){16}}
      \multiput(3,3)(4,0)3{$\bullet$}
      \put(10,4){\oval(4,4)[t]}
      \put(4,4){\line(0,1){7}}
    \end{picture}
  }
  \put(23,15){\line(-1,-1){9}}
  \put(21,7){\makebox(0,0){\bf\footnotesize (B)}}
  \put(15,13){\rotatebox{45}{\makebox[0mm]{$<_{\rm arc}$}}}
  \put(29,15){\line(1,-1){9}}
  \put(31,7){\makebox(0,0){\bf\footnotesize (D)}}
  \put(36,12){\rotatebox{-45}{\makebox[0mm]{$>_{\rm arc}$}}}
\end{picture}
\quad
\begin{picture}(72,28)(0,-5)
  \put(0,0){\begin{picture}(16,8)(0,5)
      \put(0,4){\line(1,0){16}}
      \multiput(3,3)(8,0)2{$\bullet$}
      \put(8,4){\oval(8,8)[t]}
    \end{picture}
  }
  \put(0,20){\begin{picture}(16,8)(0,5)
      \put(0,4){\line(1,0){16}}
      \multiput(3,3)(8,0)2{$\bullet$}
      \put(4,4){\line(0,1){7}}
      \put(12,4){\line(0,1){7}}
    \end{picture}
  }
  \put(8,15){\line(0,-1){9}}
  \put(3,7){\makebox(0,0){\bf\footnotesize (E)}}
  \put(10,9){\rotatebox{90}{\makebox[0mm]{$<_{\rm arc}$}}}
\end{picture}
$$

If the arc diagrams $\Delta$ and $\Delta'$ are in relation, we write $\Delta\arcleq \Delta'$.
Since each move either decreases the
number of crossings or the number of poles, $\leq_{\sf arc}$ is a partial order.
The arc-moves (A), (B), (C) and (D) were introduced in \cite{kos-sch} 
and the arc-move (E) is new.

Formally, an arc diagram is a finite set of arcs and poles in the
Poincar\'e half plane.  We assume that all end points are natural numbers
(arranged from right to left) and permit multiple arcs and poles.
In addition, each natural number can be marked by one or several circles
(which are not affected by any of the arc moves).

\begin{ex}
The diagrams $\Delta$ and $\Delta'$ are in arc-order via a single move of type {\bf (B)}.
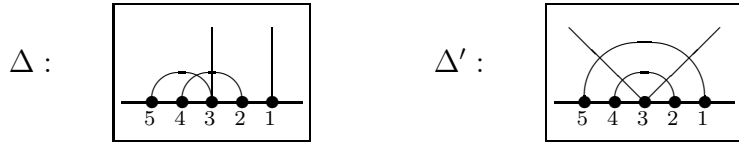
$$\raisebox{5ex}[0mm]{$\Delta:$}
  \qquad
  \framebox(26,18)[t]{\begin{picture}(24,12)(0,5)
      \put(0,4){\line(1,0){24}}
      \multiput(3,3)(4,0)5{$\bullet$}
      \put(3,1){$\scriptstyle 5$}
      \put(7,1){$\scriptstyle 4$}
      \put(11,1){$\scriptstyle 3$}
      \put(15,1){$\scriptstyle 2$}
      \put(19,1){$\scriptstyle 1$}
      \put(8,4){\oval(8,8)[t]}
      \put(12,4){\oval(8,8)[t]}
      \put(12,4){\line(0,1){10}}
      \put(20,4){\line(0,1){10}}
  \end{picture}}
  \qquad\qquad\raisebox{5ex}[0mm]{$\Delta':$}
  \qquad
  \framebox(26,18)[t]{\begin{picture}(24,12)(0,5)
      \put(0,4){\line(1,0){24}}
      \multiput(3,3)(4,0)5{$\bullet$}
      \put(3,1){$\scriptstyle 5$}
      \put(7,1){$\scriptstyle 4$}
      \put(11,1){$\scriptstyle 3$}
      \put(15,1){$\scriptstyle 2$}
      \put(19,1){$\scriptstyle 1$}
      \put(12,4){\oval(16,16)[t]}
      \put(12,4){\oval(8,8)[t]}
      \put(12,4){\line(1,1){10}}
      \put(12,4){\line(-1,1){10}}
  \end{picture}}
  $$
\end{ex}

\subsection{Invariant subspaces of nilpotent linear operators}
%-----------------------------------------------------

Let $k$ be a field.  We call a $k[T]$-module
\begin{equation}
 N_\alpha=N_\alpha(k)=\bigoplus_{i=1}^s k[T]/(T^{\alpha_i}),
 \label{eq-N-alpha}
\end{equation}
where $\alpha=(\alpha_1\geq\cdots\geq\alpha_s)$ is a partition, 
a nilpotent linear operator.
Thus, $N_\alpha$ is a representation of the quiver
$$
\hbox{\beginpicture
\setcoordinatesystem units <0.5cm,0.5cm>
%==========================================
\put{} at -4 -1
\put{} at 1 1
\put{$\circ$} at 0 0
\circulararc 150 degrees from -1.2 0 center at -.7 0
\circulararc -160 degrees from -1.2 0 center at -.7 0
\put{$\varphi$} at -1.8 0
\arr{-.31 .3} {-.2 .2}
\endpicture}
$$
subject to the relation $\varphi^{\alpha_1}=0$.
By $\mathcal S$ we denote the category of all representations $M$ of the quiver
$$
\hbox{\beginpicture
\setcoordinatesystem units <0.5cm,0.5cm>
%==========================================
\put{} at -4 -1
\put{} at 3.2 1
\put{$\circ$} at 0 0
\put{$\circ$} at  2 0
\circulararc 160 degrees from 3.2 0  center at 2.7 0
\circulararc -150 degrees from  3.2 0 center at 2.7 0
\circulararc 150 degrees from -1.2 0  center at -.7 0
\circulararc -160 degrees from -1.2 0 center at -.7 0
\arr{.4 0}{1.6 0}
\put{$\ssize \varphi$} at 3.6 0
\put{$\ssize \varphi'$} at -1.6 0
\put{$\ssize \psi$} at 1 0.4
\put{$\ssize 1$} at .2 -.5
\put{$\ssize 2$} at 1.8 -.5
\arr{2.31 0.3} {2.2 0.2}  
\arr{-.31 0.3} {-.2 0.2}
\put{$Q:$} at -4 0
\endpicture}
$$
subject to the following conditions:  The endomorphisms $M(\varphi')$ and
$M(\varphi)$ are nilpotent; the linear map $M(\psi)$ is a monomorphism;
and the relation $\varphi\psi=\psi\varphi'$ is satisfied.
Equivalently, $\mathcal S$ is the category of all monomorphisms between
nilpotent linear operators, or, the category of all short exact sequences
of nilpotent linear operators.

Given three partitions $\alpha,\beta,\gamma$, we denote by $\mathcal{S}_{\alpha,\gamma}^\beta$ 
the full subcategory  of $\mathcal{S}$ consisting of 
all short exact sequences $0\to N_\alpha\to N_\beta\to N_\gamma\to 0$
and consider the full subcategories of $\mathcal{S}$  with objects: 

\begin{equation}
 {\mathcal S}_\gamma^\beta=\bigcup_\alpha {\mathcal S}_{\alpha,\gamma}^\beta \; ,\;\;\;\; {}_2{\mathcal S}_\gamma^\beta=\bigcup_{\alpha;\, \alpha_1\leq 2} {\mathcal S}_{\alpha,\gamma}^\beta.
\label{eq-categories-S}
 \end{equation}

We assign to each object $M$ in $_2\mathcal S_\gamma^\beta$ an arc 
diagram $\Delta(M)$ which depends only on the isomorphism type of $M$.
By $_2\mathcal D_\gamma^\beta$ we denote the set of all arc diagrams
of type $(\beta,\gamma)$.
Given $\Delta\in{_2\mathcal D_\gamma^\beta}$, we denote by $\alpha(\Delta)$
the partition which contains a part 2 for each arc, and a part 1 for each 
pole.  Thus, a short exact sequence with arc diagram $\Delta$
has the form $0\to N_{\alpha(\Delta)}\to N_\beta\to N_\gamma\to 0$.

Let $k$ be an algebraically closed field, 
let $d=(a,b)\in\mathbb N^2$ be a dimension vector with $b\geq a$,
and let $\beta$ and $\gamma$ be partitions of length $b$ and $b-a$,
respectively.

\medskip
In this paper we are interested in the subspace $_2\mathbb V_\gamma^\beta$
of the representation space $\rep_d(Q)$ consisting of 
all representations which are isomorphic to an object in 
$_2\mathcal S_\gamma^\beta$.

\medskip
The space $_2\mathbb V_\gamma^\beta$ is invariant under the operation of 
the algebraic group $G=Gl_a(k)\times Gl_b(k)$ acting on $\rep_d(Q)$;
the $G$-orbits in $_2\mathbb V_\gamma^\beta$ are in one-to-one correspondence
with the isomorphism classes in $_2\mathcal S_\gamma^\beta$.
For $\Delta\in{_2\mathcal D}_\gamma^\beta$ we denote by $\mathbb V_\Delta$ the
$G$-orbit of all representations which have arc diagram $\Delta$.

\subsection{Main results}
%------------------------

The arc-diagrams of invariant subspaces give a~stratification of ${}_2{\mathbb V}_{\gamma}^\beta$:
$$
{}_2{\mathbb V}_{\gamma}^\beta=\bigcup_\Delta\mathbb{V}_\Delta.
$$

More precisely, we have:

\begin{prop}
The set $\{\mathbb V_\Delta:\Delta\in\mathcal D\}$
where $\mathcal D= {_2\mathcal D_\gamma^\beta}$ forms a
stratification for $\mathbb V={_2\mathbb V}_\gamma^\beta$ in the sense that 
\begin{enumerate}
\item each $\mathbb V_\Delta$ is locally closed in $\mathbb V$;
\item $\mathbb V$ is the disjoint union 
$\bigcup_{\Delta\in\mathcal D}^\bullet \mathbb V_\Delta$;
\item and for each $\Delta$ there is a finite subset $U_\Delta\subset\mathcal D$
such that the closure $\overline{\mathbb V}_\Delta$ is just the union
$\bigcup_{\Gamma\in U_\Delta} \mathbb V_\Gamma$.
\end{enumerate}
\end{prop}

Moreover an~arc diagram determines the dimension of its stratum as follows.

\begin{thm}\label{thm-second-main}
Let $k$ be an algebraically closed field, and let $\beta,\gamma$ be partitions.
Suppose the arc diagram 
$\Delta\in{_2\mathcal D_\gamma^\beta}$
has $x(\Delta)$ crossings. Then
$$\dim\mathbb{V}_\Delta\;=\;|\beta|^2+|\alpha|^2-n(\alpha)-n(\beta)-n(\gamma)-|\beta|-x(\Delta),$$
where 
$\alpha=\alpha(\Delta)$ 
and where $n(\lambda)=\sum_{i}\lambda_i(i-1)$ denotes the {\bf moment} and $|\lambda|=\lambda_1+\lambda_2+\ldots$ - the {length}
of a~partition $\lambda$.
\end{thm}

The second main result of this paper describes the boundary of 
each stratum.

\begin{thm}\label{thm-first-main}
Suppose that $k$ is an algebraically closed field and that $\beta,\gamma$
are partitions.
For arc diagrams $\Gamma,\Delta\in{_2\mathcal D_\gamma^\beta}$ we have
$$
\Gamma\in U_\Delta
\qquad\text{if and only if}\qquad
\Delta\leq_{\sf arc}\Gamma.$$
\end{thm}

In the proof of the ``only if'' part, we use a version of an algorithm from \cite{kos-sch} which
delivers a sequence of arc moves that convert $\Gamma$ to $\Delta$.

\begin{rem}
\begin{enumerate}
 \item In \cite{kos-sch} we investigated arc- and deg-orders in sets ${\mathbb V}_{\alpha,\gamma}^\beta(k)$, where
$\alpha_1\leq 2$. In this paper we work in a~bigger set ${}_2{\mathbb V}_{\gamma}^\beta$ and prove similar 
results without  fixing the partition $\alpha$. 
Moreover, in Section \ref{section-non-examples} we present examples 
showing that we have to fix the partitions $\beta$ and $\gamma$.
\item Working with arc-diagrams, we have to assume that $\alpha_1\leq 2$, because otherwise we do not
have the natural 
bijection between arc diagrams and orbits of our group action.  
\item
In the paper we investigate the degeneration order induced by this group action and compare 
this order with the arc-order and with two algebraic orders (hom-order and ext-order),
see Section \ref{sec-four-partial} for definitions.
\end{enumerate}

\end{rem}

%\medskip

\subsection{Organization of this paper}
%--------------------------------------

The paper is organized as follows.
\begin{itemize}
 \item In Section~\ref{sec-definitions} we introduce 
   terminology and describe a~bijection between arc diagrams and isomorphism
     types of objects in the category $_2{\mathcal S}_{\gamma}^\beta$. 
\item In Section \ref{sec-partial-orders} we recall the definition  of the partial orders 
$\leq_{\sf arc}$, $\leq_{\sf ext}$, $\leq_{\sf deg}$ and $\leq_{\sf hom}$. 
Moreover, in Theorem~\ref{thm-main} we discuss relations between them.
\item In Section \ref{section-dim-hom} we give a brief exposition of the category  $\mathcal S_2(k)$. 
\item Some technical facts for later use are presented in Section \ref{section-technical}.
\item Section \ref{chapter-hom-tab} contains the proof of the implication 
$$ \leq_{\sf hom} \quad\Longrightarrow \quad\leq_{\sf arc}$$
that is the crucial part of the proof of Theorem \ref{thm-first-main}.
Our proof uses methods derived from \cite{kos-sch}.
\item In Section \ref{section-poset-arc} some combinatorial properties of the poset $({}_{2}\mathcal{S}_\gamma^\beta,\leq_{\sf arc})$
are presented.
\item In Theorems \ref{thm-first-main} and \ref{thm-main} we assume that two 
partitions $\beta$ and $\gamma$ are fixed. 
Some relevant counterexamples indicated in Section \ref{section-non-examples} show that 
this assumption is necessary. 
\item Section \ref{sec-dim-orbit} contains a~proof of Theorem \ref{thm-second-main}.
\end{itemize}

\section{Arc diagrams and invariant subspaces}\label{sec-definitions}
%=================================

For fixed partitions $\gamma\subseteq\beta$, there is a~bijection between 
the set of arc diagrams and the set of isomorphism classes of objects in the category $_2{\mathcal S}_{\gamma}^\beta$
(\ref{eq-categories-S}).
In this section we describe this bijection.

In \cite{kos-sch} Littlewood-Richardson tableaux and Klein tableaux were applied to
associate an arc diagram with an object in $\mathcal{S}_2$, 
where $\mathcal{S}_2=\bigcup_{\gamma,\beta}{}_2\mathcal{S}_{\gamma}^\beta$. 
For the definition of a~Klein tableau for an object in $\mathcal S_2$
we refer to \cite{klein} or \cite{sch}.

\subsection{Invariant subspaces}
%-------------------------------

Let $k$ be an~arbitrary field. 
By $\mathcal N$ or $\mathcal{N}(k)$
we  denote the category of all
nilpotent linear operators $N_\alpha$ (\ref{eq-N-alpha}).
We write the objects of $\mathcal N$ as pairs $(V,\varphi)$
where $V$ is the underlying $k$-vector space and $\varphi:V\to V$
the nilpotent $k$-linear endomorphism given by multiplication by $T$.
If $(V,\varphi)$, $(V',\varphi')$ are objects in $\mathcal{N}$, then a~morphism
$f:(V,\varphi)\to (V',\varphi')$
in $\mathcal N$
is a~linear map $f:V\to V'$ such that $\varphi'f=f\varphi$.
Let 
\begin{equation}
N_\alpha=(k^{|\alpha|},\varphi_\alpha), 
\end{equation}
where $\varphi_\alpha$ is given by the nilpotent block matrix consisting of Jordan blocks of type $\alpha$.

%\smallskip
Denote by $\mathcal S$ the category 
of all systems $f=(N_\alpha,N_\beta,f)$,
where $f$ is a~monomorphism. For $f=(N_{\alpha},N_{\beta},f)$ and $g=(N_{\alpha^\prime},N_{\beta^\prime},g)$ a morphism $H:f\rightarrow g$ is a pair $(h_1,h_2)$ of homomorphisms $h_1:N_{\alpha}\rightarrow N_{\alpha^\prime}$ and $h_2:N_{\beta}\rightarrow N_{\beta^\prime}$ such that $$g\circ h_1=h_2\circ f.$$

%\smallskip
For a natural number $n$,
we write $\mathcal S_n$ or $\mathcal S_n(k)$ for the full subcategory of $\mathcal S$
of all systems where the operator acts on the subspace with nilpotency index at most $n$.
Thus, the objects in $\mathcal S_2$ are the systems $(N_\alpha,N_\beta,f)$ where $\alpha_1\leq 2$.

\subsection{Pickets and bipickets}
%---------------------------------
\label{section-pickets}

The  category $\mathcal{S}_2(k)$ is of particular interest for us
in this paper. Each indecomposable object is
either isomorphic to a~{\it picket} that is, it has the form
$$P_\ell^m=(N_{(\ell)},N_{(m)},\iota)$$
where $0\leq\ell\leq\min\{2,m\}$ (so the ambient space $N_{(m)}$ has only one Jordan block,
and $N_{(\ell)}$ is the unique $T$-invariant subspace of dimension $\ell$),
or to a  {\it bipicket}
$$B_2^{m,r}=(N_{(2)},N_{(m,r)},\delta)$$
where $1\leq r\leq m-2$
and $\delta:k[T]/(T^2)\to k[T]/(T^m)\oplus k[T]/(T^r)$ is given by
$\delta(1)=(T^{m-2},T^{r-1})$.
Whenever we want to emphasize the dependence on the field $k$, we will write $P_\ell^m=P_\ell^m(k)$ and
$B_2^{m,r}=B_2^{m,r}(k)$.

\medskip

Thanks to this classification we can associate with any object in $\mathcal{S}_2(k)$ an arc diagram 
and partition $(\beta,\gamma)$ type as follows.
First we list some arc diagrams in the following table.

\begin{center}
\begin{tabular}{|c|c|c|c|c|c|}\hline
\multicolumn6{|c|}
             {\raisebox{-1ex}[0mm]{\bf Arc diagrams
                  for the objects in $\ind\mathcal S_2$}}
             \\[2ex] \hline
$X:$\raisebox{1ex}{\phantom!}\raisebox{-1.5ex}{\phantom!} & $P_0^m$ & $P_1^m$ & $P_2^m$ & $P_2^{m}\oplus P_0^{m-1}$& $\quad\quad B_2^{m,r}\quad \quad$ 

     \\ \hline
$\Delta(X):$ & $\emptyset$ &
             $\begin{picture}(8,9)(0,3)
             \put(0,4){\line(1,0){8}}
             \put(3,3){$\bullet$}
             \put(4,4){\line(0,1){5}}
             \put(4,1){\makebox[0mm]{$\scriptstyle m$}}
             \end{picture}$
          &
             $\begin{picture}(8,9)(0,3)
             \put(0,4){\line(1,0){8}}
             \put(3,3){$\bullet$}
				 \put(4,4){\beginpicture
					\circulararc 360 degrees from 0 0 center at 0 -6.5
					\endpicture}
             %\put(4,4){\line(0,1){5}}
             \put(4,1){\makebox[0mm]{$\scriptstyle m$}}
             \end{picture}$
			&
             $\begin{picture}(16,9)(0,3)
             \put(0,4){\line(1,0){16}}
             \put(3,3){$\bullet$}
             \put(11,3){$\bullet$}
             \put(8,4){\oval(8,8)[t]}
             \put(4,1){\makebox[0mm]{$\scriptstyle m$}}
             \put(12,1){\makebox[0mm]{$\scriptstyle m-1$}}
             \end{picture}$
          &
             $\begin{picture}(16,9)(0,3)
             \put(0,4){\line(1,0){16}}
             \put(3,3){$\bullet$}
             \put(8,4){\makebox[0mm]{$\cdots$}}
             \put(11,3){$\bullet$}
             \put(8,4){\oval(8,8)[t]}
             \put(4,1){\makebox[0mm]{$\scriptstyle m$}}
             \put(12,1){\makebox[0mm]{$\scriptstyle r$}}
             \end{picture}$\\[2ex]  \hline
 {\scriptsize partition $(\beta,\gamma)$}
&
{\scriptsize $\beta=(m)$}
&
{\scriptsize $\beta=(m)$}
&
{\scriptsize $\beta=(m)$}
&
{\scriptsize $\beta=(m,m-1)$}
&
{\scriptsize $\beta=(m,r)$}
\\
 {\scriptsize  type of $X$}
&
{\scriptsize $\gamma=(m)$}
&
{\scriptsize $\gamma=(m-1)$}
&
{\scriptsize{$\gamma=(m-2)$}}
&
{\scriptsize $\gamma=(m-1,m-2)$}
&
{\scriptsize $\gamma=(m-1,r-1)$}
\\ \hline
\end{tabular}
\end{center}

  The arc diagram for an~object
$X$ in $\mathcal S_2$ is created in the following way.
We decompose $$X=Y\oplus (P_2^{s_1}\oplus P_0^{s_1-1})\oplus\ldots\oplus(P_2^{s_r}\oplus P_0^{s_r-1})\oplus P_2^{m_1}\oplus\ldots
 \oplus P_2^{m_k}\oplus P_0^{w_1}\oplus\ldots\oplus P_0^{w_l},$$

where $Y$ does have neither summands of type $P_2^m$ nor of type $P_0^m$,
and $w_1\ldots,w_l\not\in\{m_1-1,\ldots,m_k-1\}$.
For any indecomposable summand of $Y$ and any summand  
$(P_2^{s_1}\oplus P_0^{s_1-1}),\ldots,(P_2^{s_r}\oplus P_0^{s_r-1}), P_2^{m_1},\ldots, 
P_2^{m_k}, P_0^{w_1},\ldots, P_0^{w_l}$ we create the arc diagram using the above table.
The arc diagram $\Delta(X)$ of $X$ is the union of all these arc diagrams.

Similarly we define the {\it type} of $X$ as the union of partitions $\beta$ and $\gamma$ 
described in the table above.

\smallskip
\noindent{\it Example:} For the object
$$X\;=\;B_2^{5,3}\oplus B_2^{4,2}\oplus P_2^5 \oplus P_0^2\oplus P_2^3 \oplus P_1^3\oplus P_0^1\oplus P_1^1,$$
the arc diagram is

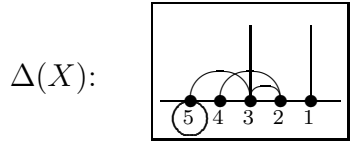
\hfil\framebox(26,18)[t]{\begin{picture}(24,12)(0,5)
    \put(-20,6){$\Delta(X)$:}
      \put(0,4){\line(1,0){24}}
      \multiput(3,3)(4,0)5{$\bullet$}
      \put(3,1){$\scriptstyle 5$}
      \put(7,1){$\scriptstyle 4$}
      \put(11,1){$\scriptstyle 3$}
      \put(15,1){$\scriptstyle 2$}
      \put(19,1){$\scriptstyle 1$}
      \put(8,4){\oval(8,8)[t]}
      \put(12,4){\oval(8,8)[t]}
		\put(14,4){\oval(4,4)[t]}
      \put(12,4){\line(0,1){10}}
      \put(20,4){\line(0,1){10}}
		\put(4,4){\beginpicture
					\circulararc 360 degrees from 0 0 center at 0 -6.5
					\endpicture}
  \end{picture}}
\hfil

and the type of $X$ is $(\beta,\gamma)=((5,5,4,3,3,3,2,2,1,1),(4,3,3,2,2,2,1,1,1)).$

\begin{rem}
  Note that two non-isomorphic object in $\mathcal S_2$ can have the same arc diagram, e.g.
 $B_2^{5,2}$ and $B_2^{5,2}\oplus P_0^3$. However, if we fix partitions $\gamma$ and $\beta$,
  the arc diagram uniquely determines the object it is associated with.
\end{rem}

\begin{defin}We denote by $_2\mathcal D_\gamma^\beta$ the set of all arc diagrams which arise 
from objects in $_2\mathcal S_\gamma^\beta$.
\end{defin}

We obtain as a consequence:

\begin{thm}\label{thm-pic-bipic}
For any field $k$ and for fixed partitions $\gamma\subseteq\beta$,
there is a one-to-one correspondence between the set of isomorphism classes
of objects in ${}_2{\mathcal S}_{\gamma}^\beta$ and the set of arc diagrams 
in $_2\mathcal D_\gamma^\beta$. \qed
\end{thm}

\begin{defin}
We say two invariant subspaces $Y,Z\in {}_2{\mathcal S}_{\gamma}^\beta$
are in {\it arc-order,} in symbols $Y\arcleq Z$, if $\Delta(Y)\arcleq \Delta(Z)$
holds.
\end{defin}

\section{The degeneration order}\label{sec-partial-orders}
%======================================

Assume that $k$ is an algebraically closed field.
For natural numbers $a\leq b$
the representation space of the quiver $Q$ corresponding to the dimension
vector $(a,b)$ is the affine variety
$${}^a{\mathbb V}^b=\mathbb{M}_a(k)\times \mathbb{M}_{a\times b}(k)\times \mathbb{M}_b(k),$$
where $\mathbb{M}_{a\times b}(k)$ is the set of $a\times b$ matrices with coefficients in $k$
and $\mathbb{M}_{a}(k)=\mathbb{M}_{a\times a}(k)$.
We work with the Zariski topology and with the  induced topology for all subsets of ${}^a{\mathbb V}^b$. 
Fix partitions $\gamma\subseteq\beta$ such that 
$|\beta|=b$ and $|\gamma|=b-a$.
We define ${\mathbb V}_\gamma^\beta$ as the subset of  ${}^a{\mathbb V}^b$ 
consisting of all points $F=(f_1,f,f_2)$, 
such that $f_1^a=0,\, f_2^b=0,\, ff_1-f_2 f=0$, $f_2$ has the Jordan type $\beta$, 
$f$ has maximal rank and $\Coker f$ has type $\gamma$. Moreover, by $\mrep$ denote the subset
of ${\mathbb V}_\gamma^\beta$ consisting of all points $(f_1,f,f_2)$ such that $f_1$ has 
nilpotency index less than or equal to $2$,
that is, the Jordan blocks in the Jordan normal form for $f_1$ have size at most $2$.
On ${\mathbb V}_\gamma^\beta$ (resp. $\mrep$), there acts the algebraic group 
$G=Gl(a,b)=Gl(a)\times Gl(b)$ via $(g,h)\cdot (f_1,f,f_2)=(g f_1 g^{-1}, hfg^{-1}, hf_2h^{-1})$. 

For a~point $F\in {\mathbb V}_\gamma^\beta$, denote by $\mathcal{O}_F$ the orbit of $F$ under action of $G$. \smallskip 
 
\begin{rem}
  Note that the $G$ - orbits in ${\mathbb V}_\gamma^\beta$ (resp.\ in $\mrep$) are in $1 - 1$ - correspondence with the isomorphism classes of objects 
 in ${\mathcal S}_\gamma^\beta$ (resp. ${}_2{\mathcal S}_\gamma^\beta$).
\end{rem}\smallskip

\begin{defin}
Suppose that the points $F, H\in {}_2{\mathbb V}_{\gamma}^\beta(k)$ correspond to
objects 
$Y=(N_\alpha,N_\beta,f)$ and $Z=(N_{\tilde \alpha},N_\beta,h)$ 
in ${}_2{\mathcal S}_{\gamma}^\beta(k)$.

The relation $Y \leq_{\sf deg} Z$ is defined to 
hold if and only if  $\mathcal{O}_H \subseteq \overline{\mathcal{O}_F}$ 
where $\overline{\mathcal{O}_F}$ is the closure of $\mathcal{O}_F$ in ${}_2{\mathbb V}_{\gamma}^\beta(k)$.
\end{defin}

\subsection{The algebraic orders}\label{sec-four-partial}
%-------------------------------

%\smallskip
Our aim is to prove that the degeneration order and the arc order are equivalent.
In the proof we use the following classical algebraic orders.
Let $Y=(N_\alpha,N_\beta,f)$ and $Z=(N_{\tilde \alpha},N_\beta,g)$ be objects
in $\mathcal S_{\gamma}^\beta(k)$.

\begin{itemize}
 \item The relation $Y \leq_{\sf ext} Z$  holds if there exist
a natural number $s$,
objects  $M_i$, $U_i$, $V_i$ in $\mathcal{S}(k)$ and short exact
sequences $0\to U_i\to M_i\to V_i\to 0$ in~$\mathcal{S}(k)$
such that $Y\cong M_1$, $U_i\oplus V_i\cong M_{i+1}$ for $1\leq i\leq s$,
and $Z\cong M_{s+1}$.

\item The relation $Y\leq_{\sf hom} Z$  holds
if $$[X,Y]\leq [X,Z]$$ for any object $X$ in $\mathcal{S}(k)$.
Here we write $[X,Y]=\dim_k\Hom_{\mathcal{S}}(X,Y)$ for $\mathcal{S}$-modules $X,Y$.
By \cite[Lemma 3.3]{kos-sch}, the relation $Y\leq_{\sf hom} Z$  holds
if $[X,Y]\leq [X,Z]$ for every object $X$ in $\mathcal{S}_2(k)$.

\end{itemize}

\subsection{The partial orders are equivalent}
%---------------------------

We can present the proof of Theorem~\ref{thm-first-main}, up to two results about the arc-order which are
shown in the next section.

%\smallskip
We restate the theorem to include statements about arbitrary fields.

\begin{thm}\label{thm-main} Let $k$ be an~arbitrary field and assume that
$Y,Z\in {}_2{\mathcal S}_{\gamma}^\beta$.
The following conditions are equivalent
\begin{enumerate}
 \item $Y\leq_{\sf arc} Z$,
 \item $Y\leq_{\sf ext} Z$,
 \item $Y\leq_{\sf hom} Z$.
\end{enumerate}
If in addition the field $k$ is algebraically closed,
then the conditions stated above are equivalent with
\begin{enumerate}
 \item[(4)] $Y\leq_{\sf deg} Z$.
\end{enumerate}
\end{thm}

\begin{proof}
Applying the functor
$\Hom_k(X,-)$ to the short exact sequences given in the definition of $\leq_{\sf ext}$,
it is easy to see that $Y\leq_{\sf ext} Z$ implies $Y\leq_{\sf hom} Z.$

If $k$ is an~algebraically closed field, then by \cite{bongartz,riedtmann} (see \cite[Section 3]{kos-sch} for details),
%Lemmata \ref{lem-deg-order}, \ref{lem-ext-order} and \ref{lem:hom-order}
we have
$$Y\leq_{\sf ext} Z \;\Longrightarrow\; Y\leq_{\sf deg} Z\;\Longrightarrow\; Y\leq_{\sf hom} Z.$$
For any field $k$, the implications
$$Y\leq_{\sf hom} Z \;\Longrightarrow\; Y\leq_{\sf arc} Z\;\Longrightarrow\; Y\leq_{\sf ext} Z$$
follow from Theorem \ref{theorem-hom-arc} and Lemma \ref{lem-tab-ext}, respectively.
\end{proof}

\section{The category $\mathcal S_2(k)$}\label{section-dim-hom}
%------------------------------------------

In this section we shortly recall properties of the category $\mathcal S_2(k)$.
Denote by $\mathcal S_2^n(k)$ the full subcategory of $\mathcal S_2(k)$
of all objects where the operator acts with nilpotency index at most $n$
on the ambient space.
We have seen in \cite[Section~3.2]{s-b} that $\mathcal S_2^n(k)$ is an
exact Krull-Remak-Schmidt category with Auslander-Reiten sequences.

The Auslander-Reiten quiver for each of the categories $\mathcal S_2^n(k)$
is obtained by identifying the objects of type $P_1^r$ on the left with their
counterparts on the
right in the following picture, thus yielding a Moebius band.

$$
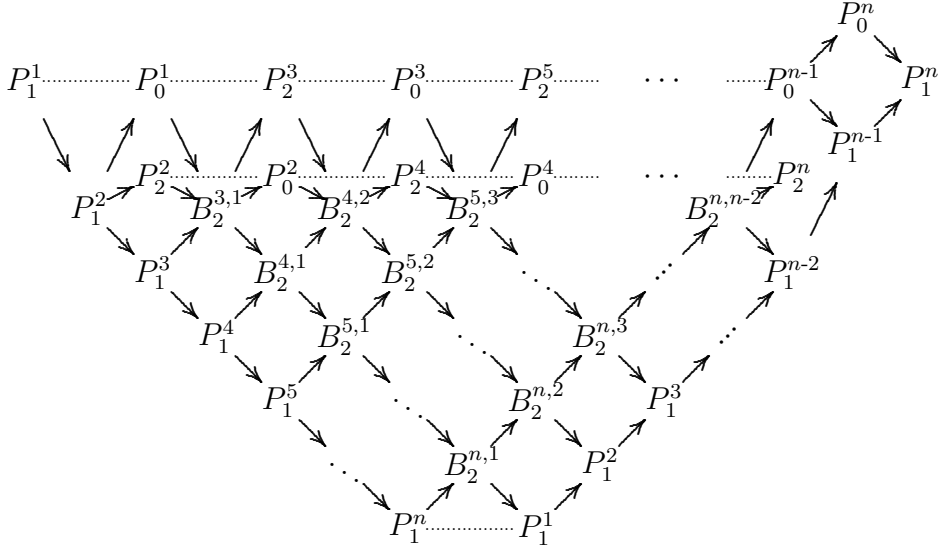
\beginpicture\setcoordinatesystem units <.85cm,.85cm>
\setcounter{boxsize}{2}
\put {} at 0 6.5
\put {} at 0 -.8
%\put{$\Gamma_{\mathcal S_2}:$} at 0.5 2
\setlength\unitlength{1mm}
\put{$P_1^1$} at 0 6
%\put{$\times$} at -.5 6
\put{$P_1^2$} at 1 4
%\put{$\times$} at .5 4
\put{$P_0^1$} at 2 6
\put{$P_2^2$} at 2 4.5
\put{$P_1^3$} at 2 3
%\put{$\times$} at 1.5 3
\put{$B_2^{3,1}$} at 3 4
\put{$P_1^4$} at 3 2
%\put{$\times$} at 2.5 2
\put{$P_2^3$} at 4 6
\put{$P_0^2$} at 4 4.5
%\put{{\bf (*)}} at 4 4
\put{$B_2^{4,1}$} at 4 3
\put{$P_1^5$} at 4 1
%\put{$\times$} at 3.5 1
\put{$B_2^{4,2}$} at 5 4
\put{$B_2^{5,1}$} at 5 2
%\put{$P_1^6$} at 5 0
%\put{$\times$} at 4.5 0
\put{$P_0^3$} at 6 6
\put{$P_2^4$} at 6 4.5
\put{$B_2^{5,2}$} at 6 3
%\put{$B_2^{6,1}$} at 6 1
\put{$P_1^n$} at 6 -1
\put{$B_2^{5,3}$} at 7 4
%\put{$B_2^{6,2}$} at 7 2
\put{$B_2^{n,1}$} at 7 0
\put{$P_2^5$} at 8 6
\put{$P_0^4$} at 8 4.5
%\put{$B_2^{6,3}$} at 8 3
\put{$B_2^{n,2}$} at 8 1
\put{$P_1^1$} at 8 -1
%\put{$B_2^{6,4}$} at 9 4
\put{$B_2^{n,3}$} at 9 2
\put{$P_1^2$} at 9 0
%\put{$P_2^6$} at 10 4.5
%\put{$P_0^5$} at 10 6
\put{$P_1^3$} at 10 1
%\put{$\cdots$} at 11 6
%\put{$P_1^4$} at 11 2
\put{$B_2^{n,n\text{-}2}\;$} at 11 4
\put{$P_1^{n\text-2}$} at 12 3
\put{$P_2^n$} at 12 4.5
\put{$P_0^{n\text-1}$} at 12 6
\put{$P_0^n$} at 13 7
\put{$P_1^{n\text{-}1}$} at 13 5
\put{$P_1^n$} at 14 6
%\multiput{$\cdot$} at 12.9 3.9  13 4  13.1 4.1 /
\arr{.3 5.4}{.7 4.6}
\arr{1.3 4.6}{1.7 5.4}
\arr{1.3 4.15}{1.7 4.35}
\arr{1.3 3.7}{1.7 3.3}
\arr{2.3 5.4}{2.7 4.6}
\arr{2.3 4.35}{2.7 4.15}
\arr{2.3 3.3}{2.7 3.7}
\arr{2.3 2.7}{2.7 2.3}
\arr{3.3 4.6}{3.7 5.4}
\arr{3.4 4.2}{3.7 4.35}
\arr{3.3 3.7}{3.7 3.3}
\arr{3.3 2.3}{3.7 2.7}
\arr{3.3 1.7}{3.7 1.3}
\arr{4.3 5.4}{4.7 4.6}
\arr{4.3 4.35}{4.7 4.15}
\arr{4.4 3.4}{4.7 3.7}
\arr{4.3 2.7}{4.7 2.3}
\arr{4.3 1.3}{4.7 1.7}
\arr{4.3 .7}{4.7 .3}
\arr{5.3 4.6}{5.7 5.4}
\arr{5.4 4.2}{5.7 4.35}
\arr{5.3 3.7}{5.7 3.3}
\arr{5.3 2.3}{5.7 2.7}
\arr{5.3 1.7}{5.7 1.3}
%\arr{5.3 .3}{5.7 .7}
\arr{5.3 -.3}{5.7 -.7}
\arr{6.3 5.4}{6.7 4.6}
\arr{6.3 4.35}{6.7 4.15}
\arr{6.4 3.4}{6.7 3.7}
\arr{6.3 2.7}{6.7 2.3}
%\arr{6.3 1.3}{6.7 1.7}
\arr{6.3 -.7}{6.7 -.3}
\arr{6.3 .7}{6.7 .3}
\arr{7.3 4.6}{7.7 5.4}
\arr{7.4 4.2}{7.7 4.35}
\arr{7.3 3.7}{7.7 3.3}
%\arr{7.4 2.4}{7.7 2.7}
\arr{7.3 1.7}{7.7 1.3}
\arr{7.3 .3}{7.7 .7}
\arr{7.3 -.3}{7.7 -.7}
%\arr{8.3 5.4}{8.7 4.6}
%\arr{8.3 4.35}{8.7 4.15}
%\arr{8.4 3.4}{8.7 3.7}
\arr{8.3 2.7}{8.7 2.3}
\arr{8.3 1.3}{8.7 1.7}
\arr{8.3 .7}{8.7 .3}
\arr{8.3 -.7}{8.7 -.3}
%\arr{9.4 4.8}{9.7 5.4}
%\arr{9.4 4.2}{9.7 4.35}
%\arr{9.3 3.7}{9.7 3.3}
\arr{9.3 2.3}{9.7 2.7}
\arr{9.3 1.7}{9.7 1.3}
\arr{9.3 .3}{9.7 .7}
%\arr{10.3 5.5}{10.7 4.7}
\arr{10.3 3.3}{10.7 3.7}
\arr{10.3 1.3}{10.7 1.7}
\arr{11.5 4.25}{11.7 4.35}
\arr{11.3 4.6}{11.7 5.4}
\arr{11.3 3.7}{11.7 3.3}
\arr{11.3 2.3}{11.7 2.7}
\arr{12.3 3.6}{12.7 4.4}
\arr{12.3 5.7}{12.7 5.3}
\arr{12.3 6.3}{12.7 6.7}
\arr{13.3 6.7}{13.7 6.3}
\arr{13.3 5.3}{13.7 5.7}
\setdots<2pt>
\plot 0.3 6  1.7 6 /
\plot 2.3 6  3.7 6 /
\plot 4.3 6  5.7 6 /
\plot 6.3 6  7.7 6 /
\plot 8.3 6  9 6 /
\plot 11 6  11.7 6 /
\plot 2.3 4.5  3.7 4.5 /
\plot 4.3 4.5  5.7 4.5 /
\plot 6.3 4.5  7.7 4.5 /
\plot 8.3 4.5  9 4.5 /
\plot 11 4.5  11.7 4.5 /
\plot 6.3 -1  7.7 -1 /
\multiput{$\ddots$} at 8 3  7 2  6 1  5 0 /
\multiput{$\cdots$} at 10 6  10 4.5 /
\multiput{$\cdot$} at 9.9 2.9  10 3  10.1 3.1     10.9 1.9  11 2  11.1 2.1 /
\endpicture
$$

For each pair $(X,Y)$ of indecomposable objects in $\mathcal S_2(k)$ we determine
in the table below
the dimension of the $k$-space $\Hom_{\mathcal S}(X,Y)$, see \cite[Lemma 4]{sch} and \cite{kos-sch}.

\begin{figure}[tbh]
\begin{center}
\begin{tabular}{|r|@{}c@{}|@{}c@{}|@{}c@{}|@{}c@{}|}\hline
\multicolumn{5}{|c|}{\bf Dimensions of Spaces $\Hom(X,Y)$, $X,Y\in \ind\mathcal S_2(k)$}\\
\hline
  $X$ &  $Y=P_0^m$  &  $P_2^m$  &  $B_2^{m,r}$  &  $P_1^m$  \\
\hline \hline
$P_0^\ell$ & $\min\{\ell,m\}$ & $\min\{\ell,m\}$
            & $\begin{array}{l} \min\{\ell,m\} \\ + \min\{\ell,r\} \end{array}$
            & $\min\{\ell,m\}$ \\
\hline
$P_2^\ell$      & $\min\{\ell-2,m\}$
               & $\min\{\ell,m\}$
               & $\begin{array}{l} \min\{\ell-1,m\} \\ +\min\{\ell-1,r\} \end{array}$
               & $\min\{\ell-1,m\}$ \\
\hline
$B_2^{\ell,t}$ & $\begin{array}{l}  \min\{\ell-1,m\} \\ + \min\{t-1,m\} \end{array}$
             & $\begin{array}{l}  \min\{\ell,m\} \\ + \min\{t,m\} \end{array}$
             & $\begin{array}{l}  \min\{\ell-1,m\} \\ + \min\{t,m\} \\
                        + \min\{\ell-1,r\} \\ + \min\{t,r\} \\-  {\bf 1}\{\ell>m\,\text{and}\,t\leq r\} \end{array}$
             & $\begin{array}{l} \min\{\ell-1,m\} \\ + \min\{t,m\} \end{array}$ \\
\hline
$P_1^\ell$ & $\min\{\ell-1, m\}$
          & $\min\{\ell,m\}$
          & $\begin{array}{l}  \min\{\ell,m\} \\ +  \min\{\ell-1,r\} \end{array}$
          & $\min\{\ell,m\}$ \\
\hline
\end{tabular}
\end{center}
\end{figure}

We denote by {\bf 1} the {\it characteristic function} corresponding to the property
specified in parantheses.

%\medskip
For the sake of simplicity we use the
notation
$$B_2^{m,m-1}\;=\;P_2^m\oplus P_0^{m-1}$$
for $m\geq 2$.
We observe that the notation is consistent with the formulae above.

\subsection{How operations change the hom spaces}\label{section-operations}
%------------------------------------------------

Throughout this section, $Y,Z\in\mathcal S_2$ will be objects of the same partition type $(\beta,\gamma)$.
Following \cite{kos-sch}, we introduce two matrices, the multiplicity matrix $\delta M=\delta M(Y,Z)$
and the hom matrix $\delta H=\delta H(Y,Z)$; in each case the indexing set
is the set of isomorphism types of indecomposable objects in $\mathcal S_2$.
The matrices are defined as follows:
$$\delta M_X=\mu_X(Z)-\mu_X(Y), \quad\text{and}\quad
  \delta H_X=[X,Z]_{\mathcal S}-[X,Y]_{\mathcal S},$$
where $[X,Z]_{\mathcal S}=\dim\Hom_{\mathcal S}(X,Z)$ and
where $\mu_X(Z)$ denotes the number of direct summands of $Z$ that are isomorphic to $X$.

%\medskip
We visualize the matrices by indicating the value at $X\in\ind\mathcal S_2$
in the position of $X$ in the Auslander-Reiten quiver for $\mathcal S_2^n$,
with $n$ large enough.
We sketch this quiver as follows:  The modules on the top line are
the $P_2^m$, those on the second line are the $P_0^r$;
the modules in the triangle have type $B_2^{m,r}$.
The modules $P_1^r$ are repeated twice, on the diagonal at the left
and on the antidiagonal at the right.

$$
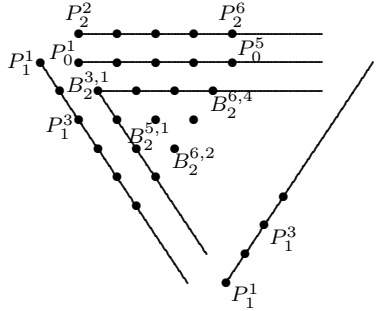
\beginpicture\setcoordinatesystem units <.85cm,.85cm>
\multiput{} at 0 0  4 4 /
\plot -.3 3.9  3.5 3.9 /
\plot -.3 3.45  3.5 3.45 /
\plot 1.7 .45  0 3  3.5 3 /
\plot 2 0  4.3 3.45 / % 2 0  3.7 2.55 /
\plot -.9 3.45  1.4 0 /
\multiput{$\scriptstyle\bullet$} at 2 0  2.3 .45  2.6 .9  2.9  1.35
                                 0 3  .3 2.55  .6 2.1  .9 1.65  .6 3  1.2 3  1.8 3
                                 .9 2.55  1.5 2.55  1.2 2.1
                                 -.3 3.45  .3 3.45  .9 3.45  1.5 3.45  2.1 3.45
                                 -.3 3.9  .3 3.9  .9 3.9  1.5 3.9  2.1 3.9
                                 -.9 3.45  -.6 3  -.3 2.55  0 2.1  .3 1.65  .6 1.2 /
\multiput{$\scriptstyle P_1^1$} at 2.3 -.2  -1.2 3.45 /
\multiput{$\scriptstyle P_1^3$} at 2.9 .7   -.6 2.45 /
\put{$\scriptstyle B_2^{3,1}$} at -.15 3.1
\put{$\scriptstyle B_2^{5,1}$} at .8 2.3
\put{$\scriptstyle B_2^{6,4}$} at 2.1 2.8
\put{$\scriptstyle B_2^{6,2}$} at 1.5 1.9
\put{$\scriptstyle P_0^1$} at -.55 3.6
\put{$\scriptstyle P_2^2$} at -.3 4.2
\put{$\scriptstyle P_0^5$} at 2.4 3.65
\put{$\scriptstyle P_2^6$} at 2.1 4.2
\endpicture
$$

%\medskip

% \begin{proof}
% For the first statement we use the table in the previous section
% to verify that the dimensions of
% homomorphism spaces are determined by the partition type:
% $$\begin{array}{rcl}
%   [P_0^m,Y]_{\mathcal S} & = & \bar\beta_1+\cdots+\bar\beta_m \\[.1ex]
%   [P_2^m,Y]_{\mathcal S} & = & \bar\alpha_1+\bar\alpha_2+\bar\gamma_1+\cdots \bar\gamma_{m-2}
% \end{array}$$
% 
% For the second assertion, let $n=\beta_1$ be the nilpotency index of the operator
% acting on $Y$.  By comparing the third and the fourth row in the table, we see
% that for each $m>n$ the equality $[B_2^{m,r},Y]_{\mathcal S}=|\beta|+[P_1^r,Y]_{\mathcal S}$
% holds where $|\beta|=\sum_i\beta_i$.
% \end{proof}

We determine how the matrices change under the operation {\bf(E)} on the arc diagram.

%
%  Case (A)
%

%\medskip
%\noindent{\bf (E)}
Suppose $Z$ is obtained from $Y$ by a transformation
$$
{\bf (E)}:\qquad
\begin{picture}(20,8)(0,3)
\put(0,4){\line(1,0){20}}
\multiput(3,3)(12,0)2{$\bullet$}
\put(10,4){\oval(12,12)[t]}
\put(4,1){\makebox[0mm]{$\scriptstyle m$}}
\put(16,1){\makebox[0mm]{$\scriptstyle r$}}
\end{picture}
\;
\leq_{\sf arc}
\;
\begin{picture}(20,8)(0,3)
\put(0,4){\line(1,0){20}}
\multiput(3,3)(12,0)2{$\bullet$}
\put(4,4){\line(0,1){10}}
\put(16,4){\line(0,1){10}}
\put(4,1){\makebox[0mm]{$\scriptstyle m$}}
\put(16,1){\makebox[0mm]{$\scriptstyle r$}}
\end{picture}
$$
where $m>r+1$.
Recall that the arc in the diagram for $Y$ represents
the bipicket $B_2^{m,r}$ which is replaced by
the summands $P_1^{m}$ and $P_1^{r}$ in $Z$ that give rise
to the corresponding poles in the diagram for $Z$.

%\smallskip
Thus, the multiplicity matrix is as follows.
$$
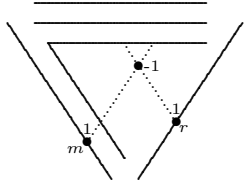
\beginpicture\setcoordinatesystem units <.6cm,.6cm>
\multiput{} at 0 0  4.5 4.5 /

\put{ } at -3 2
\plot -.3 3.9  3.5 3.9 /
\plot -.3 3.45  3.5 3.45 /
\plot 1.7 .45  0 3  3.5 3 /
\plot 2 0  4.3 3.45 /
\plot 1.4 0  -.9 3.45 /
\setdots<2pt>
\plot .86 .8  2.33 3 /
%\plot 2.5 .75  1 3 /
\plot 2.84 1.25  1.67 3 /
%\put{$\ssize s$} at 0 1.55
%\put{$\ssize r$} at .3 1.1
\put{$\ssize m$} at .6 .65
%\put{$\ssize s$} at 2.7 .6
\put{$\ssize r$} at 3.0 1.1
\multiput{$\scriptstyle\bullet$} at  .86 .8 2 2.5  2.84 1.25 /
\multiput{$\ssize 1$} at 2.84 1.55 .86 1.1 /%1.4 2  .66 1.5
\multiput{$\ssize \text-1$} at 2.3 2.5 /
\endpicture
$$

Note that the marked points correspond to a short exact sequence
$$0\longrightarrow P_1^{m}\longrightarrow B_2^{m,r}\longrightarrow P_1^{r}\longrightarrow 0$$
which proves the implication
$$Y\arcleq Z\;\Longrightarrow\;Y\extleq Z.$$

%\medskip
Next we determine the Hom matrix $\delta H=\delta H(B_2^{m,r},P_1^{m}\oplus  P_1^{r})$.
By Lemma~\ref{lemma-zeros}, we have $\delta H_{P_0^\ell}=0=\delta H_{P_2^\ell}$.
We compute the remaining numbers using the table given in the 
first part of Section~\ref{section-dim-hom}: $$\delta H_{P_1^\ell}=1, \text{ where } \ell\leq r$$
$$\delta H_{B_2^{\ell,t}}={\bf 1}\{m<\ell\,\text{and}\,t\leq r\}.$$
Thus, the only indecomposables $X\in\mathcal S_2$
for which $\delta H_X\neq 0$ are the $B_2^{\ell,t}$ where
$m<\ell$ and $t\leq r$ and $P_1^{\ell}$ where $\ell\leq r$.  For each such module $X$ we have $\delta H_X=1$.
They lie in the shaded region in the diagram below.

%\smallskip
$$
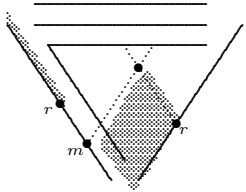
\beginpicture\setcoordinatesystem units <.6cm,.6cm>
\multiput{} at 0 0  4.5 4.5 /
\put{ } at -3 2
\plot -.3 3.9  3.5 3.9 /
\plot -.3 3.45  3.5 3.45 /
\plot 1.7 .45  0 3  3.5 3 /
\plot 2 0  4.3 3.45 /
\plot 1.4 0  -.9 3.45 /
\setdots<2pt>
\plot .86 .8  2.33 3 /
%\plot 2.5 .75  1 3 /
\plot 2.84 1.25  1.67 3 /
\put{$\ssize r$} at 0 1.55
%\put{$\ssize r$} at .3 1.1
\put{$\ssize m$} at .6 .65
%\put{$\ssize s$} at 2.7 .6
\put{$\ssize r$} at 3.0 1.1
\multiput{$\scriptstyle\bullet$} at  .86 .8 2 2.5  2.84 1.25 .26 1.7 /% .56 1.25  .26 1.7 2.5 .75 1.67 2
%\multiput{$\ssize 1$} at 1.4 2  2.84 1.55 .66 1.5 /
%\multiput{$\ssize \text-1$} at 2.3 2.5  2.2 .75  .46 1.95  /
\setshadegrid span <.3mm>
\vshade 1.1 .8 .8 <,z,,>
		  1.9 -.25 2.15  <z,z,,>
        2.2 .15 2.5  <z,z,,>
       % 2.6  .9  1.9  <z,z,,>
        2.94 1.4 1.4 /
\vshade  -1. 3.2 3.2 <,z,,>
         -.9 3.45 3.75 <z,z,,>
         .36 1.55 1.85 <z,z,,>
         .38 1.55 1.85 /
\endpicture$$

%-.9 3.45
We have seen in \cite{kos-sch}  how the matrices are changing for the operations 
{\bf (A)}-{\bf (D)} on the arc diagram. We present these matrices in the following table:

\begin{center}
\renewcommand{\arraystretch}{2}
\begin{tabular}{|r|@{} c | c |}
\hline
\multicolumn{3}{|c|}{\bf Matrices for moves {\bf(A)}-{\bf(D)}}\\
\hline
\multirow{4}{*}{{\bf(A)}} &
\multicolumn{2}{|c|}{$
\begin{picture}(20,8)(0,3)
\put(0,4){\line(1,0){20}}
\multiput(3,3)(4,0)4{$\bullet$}
\put(10,4){\oval(4,4)[t]}
\put(10,4){\oval(12,12)[t]}
\put(4,1){\makebox[0mm]{$\scriptstyle m$}}
\put(8,1){\makebox[0mm]{$\scriptstyle n$}}
\put(12,1){\makebox[0mm]{$\scriptstyle r$}}
\put(16,1){\makebox[0mm]{$\scriptstyle s$}}
\end{picture}
\;
\leq_{\sf arc}
\;
\begin{picture}(20,8)(0,3)
\put(0,4){\line(1,0){20}}
\multiput(3,3)(4,0)4{$\bullet$}
\put(8,4){\oval(8,8)[t]}
\put(12,4){\oval(8,8)[t]}
\put(4,1){\makebox[0mm]{$\scriptstyle m$}}
\put(8,1){\makebox[0mm]{$\scriptstyle n$}}
\put(12,1){\makebox[0mm]{$\scriptstyle r$}}
\put(16,1){\makebox[0mm]{$\scriptstyle s$}}
\end{picture}$, where $m>n>r>s$}
\\
\cline{2-3}&$\delta M(Y,Z)$&$\delta H(Y,Z)$
\\
\cline{2-3}
&$
\beginpicture\setcoordinatesystem units <.6cm,.6cm>
\multiput{} at 0 0  4.5 4.5 /
%\put{$\delta M(Y,Z):$} at -3 2
\put{ } at -3 2
\plot -.3 3.9  3.5 3.9 /
\plot -.3 3.45  3.5 3.45 /
\plot 1.7 .45  0 3  3.5 3 /
\plot 2 0  4.3 3.45 /
\plot 1.4 0  -.9 3.45 /
\setdots<2pt>
\plot .84 .8  2.33 3 /
\plot 1.2 .3  3 3 /
\plot 2.5 .75  1 3 /
\plot 2.84 1.25  1.67 3 /
%\put{$\ssize n$} at .6 .65
\put{$\ssize m$} at 1 .15
\put{$\ssize s$} at 2.7 .6
\put{$\ssize r$} at 3 1.1
\multiput{$\scriptstyle\bullet$} at 2 1.5  1.67 2  2.33 2  2 2.5 /
\multiput{$\ssize 1$} at 1.4 2  2.6 2 /
\multiput{$\ssize -1$} at 2 1.3  2 2.7 /
\endpicture
$
&
$
\beginpicture\setcoordinatesystem units <.6cm,.6cm>
\multiput{} at 0 0  4.5 4.5 /
%\put{$\delta H(Y,Z):$} at -3 2
\put{ } at -3 2
\plot -.3 3.9  3.5 3.9 /
\plot -.3 3.45  3.5 3.45 /
\plot 1.7 .45  0 3  3.5 3 /
\plot 2 0  4.3 3.45 /
\plot 1.4 0  -.9 3.45 /
\setdots<2pt>
\plot .84 .8  2.33 3 /
\plot 1.2 .3  3 3 /
\plot 2.5 .75  1 3 /
\plot 2.84 1.25  1.67 3 /
\put{$\ssize n$} at .6 .65
\put{$\ssize m$} at 1 .15
\put{$\ssize s$} at 2.7 .6
\put{$\ssize r$} at 3 1.1
\multiput{$\scriptstyle\bullet$} at 2 1.5  1.67 2  2.33 2  2 2.5 /
\setshadegrid span <.3mm>
\vshade 1.87 2 2 <,z,,>
        2.2 1.5 2.5  <z,z,,>
        2.53 2 2  /
\endpicture
$\\
\cline{2-3}
&\multicolumn{2}{|c|}{$0\longrightarrow B_2^{n,s}\longrightarrow B_2^{m,s}\oplus B_2^{n,r}\longrightarrow B_2^{m,r}\longrightarrow 0$}
\\
\hline
\hline
\multirow{4}{*}{{\bf(A')}} &
\multicolumn{2}{|c|}{$
\begin{picture}(20,8)(0,3)
\put(0,4){\line(1,0){20}}
\multiput(3,3)(4,0)4{$\bullet$}
\put(10,4){\oval(4,4)[t]}
\put(10,4){\oval(12,12)[t]}
\put(4,1){\makebox[0mm]{$\scriptstyle m$}}
\put(8,1){\makebox[0mm]{$\scriptstyle n$}}
\put(12,1){\makebox[0mm]{$\scriptstyle r$}}
\put(16,1){\makebox[0mm]{$\scriptstyle s$}}
\end{picture}
\;
\leq_{\sf arc}
\;
\begin{picture}(20,8)(0,3)
\put(0,4){\line(1,0){20}}
\multiput(3,3)(4,0)4{$\bullet$}
\put(8,4){\oval(8,8)[t]}
\put(12,4){\oval(8,8)[t]}
\put(4,1){\makebox[0mm]{$\scriptstyle m$}}
\put(8,1){\makebox[0mm]{$\scriptstyle n$}}
\put(12,1){\makebox[0mm]{$\scriptstyle r$}}
\put(16,1){\makebox[0mm]{$\scriptstyle s$}}
\end{picture}$, where $m>n>r>s$ and $n=r+1$}
\\
\cline{2-3}&$\delta M(Y,Z)$&$\delta H(Y,Z)$
\\
\cline{2-3}
&$
\beginpicture\setcoordinatesystem units <.6cm,.6cm>
\multiput{} at 0 0  4.5 4.5 /
\put{ } at -3 2
\plot -.3 3.9  3.5 3.9 /
\plot -.3 3.45  3.5 3.45 /
\plot 1.7 .45  0 3  3.5 3 /
\plot 2 0  4.3 3.45 /
\plot 1.4 0  -.9 3.45 /
\setdots<2pt>
\plot .54 1.3  1.67 3 /
\plot 1.2 .3  3 3 /
\plot 2.5 .75  1 3 /
\plot 2.33 3  3.16 1.75 /
\put{$\ssize n$} at .3 1.15
\put{$\ssize m$} at 1 .15
\put{$\ssize s$} at 2.7 .6
\put{$\ssize r$} at 3.4 1.6
\multiput{$\scriptstyle\bullet$} at 2 1.5  1.33 2.5  2.67 2.5  2 3.45  2 3.9 /
\multiput{$\ssize 1$} at 1.1 2.5  2.9 2.5 /
\multiput{$\ssize -1$} at 2 1.3  2.3 3.6  2.3 4.05 /
\endpicture
$
&
$
\beginpicture\setcoordinatesystem units <.6cm,.6cm>
\multiput{} at 0 0  4.5 4.5 /
\put{ } at -3 2
\plot -.3 3.9  3.5 3.9 /
\plot -.3 3.45  3.5 3.45 /
\plot 1.7 .45  0 3  3.5 3 /
\plot 2 0  4.3 3.45 /
\plot 1.4 0  -.9 3.45 /
\setdots<2pt>
\plot .54 1.3  1.67 3 /
\plot 1.2 .3  3 3 /
\plot 2.5 .75  1 3 /
\plot 2.33 3  3.16 1.75 /
\put{$\ssize n$} at .3 1.15
\put{$\ssize m$} at 1 .15
\put{$\ssize s$} at 2.7 .6
\put{$\ssize r$} at 3.4 1.6
\multiput{$\scriptstyle\bullet$} at 2 1.5  1.33 2.5  2.67 2.5  2 3.45  2 3.9 /
\setshadegrid span <.3mm>
\vshade 1.53 2.5 2.5  <,z,,>
        1.97 1.85 3.15  <z,z,,>
        2.2 1.5 3.15   <z,z,,>
        2.43 1.85 3.15  <z,z,,>
        2.87 2.5 2.5 /
\endpicture
$\\
\cline{2-3}
&\multicolumn{2}{|c|}{$0\longrightarrow B_2^{n,s}\longrightarrow B_2^{m,s}\oplus P_2^n \oplus P_0^r\longrightarrow B_2^{m,r}\longrightarrow 0$}
\\
\hline
\hline
\multirow{4}{*}{{\bf(B)}} &
\multicolumn{2}{|c|}{$\begin{picture}(16,8)(0,3)
\put(0,4){\line(1,0){16}}
\multiput(3,3)(4,0)3{$\bullet$}
\put(6,4){\oval(4,4)[t]}
\put(12,4){\line(0,1){7}}
\put(4,1){\makebox[0mm]{$\scriptstyle m$}}
\put(8,1){\makebox[0mm]{$\scriptstyle r$}}
\put(12,1){\makebox[0mm]{$\scriptstyle s$}}
\end{picture}
\;
\leq_{\sf arc}
\;
\begin{picture}(16,8)(0,3)
\put(0,4){\line(1,0){16}}
\multiput(3,3)(4,0)3{$\bullet$}
\put(8,4){\oval(8,8)[t]}
\put(8,4){\line(0,1){7}}
\put(4,1){\makebox[0mm]{$\scriptstyle m$}}
\put(8,1){\makebox[0mm]{$\scriptstyle r$}}
\put(12,1){\makebox[0mm]{$\scriptstyle s$}}
\end{picture}
$
where $m>r>s$.}
\\
\cline{2-3}&$\delta M(Y,Z)$&$\delta H(Y,Z)$
\\
\cline{2-3}
&$
\beginpicture\setcoordinatesystem units <.6cm,.6cm>
\multiput{} at 0 0  4.5 4.5 /
\put{ } at -3 2
\plot -.3 3.9  3.5 3.9 /
\plot -.3 3.45  3.5 3.45 /
\plot 1.7 .45  0 3  3.5 3 /
\plot 2 0  4.3 3.45 /
\plot 1.4 0  -.9 3.45 /
\setdots<2pt>
\plot .86 .8  2.33 3 /
\plot 2.5 .75  1 3 /
\plot 2.84 1.25  1.67 3 /
\put{$\ssize s$} at 0 1.55
\put{$\ssize r$} at .3 1.1
\put{$\ssize m$} at .6 .65
\put{$\ssize s$} at 2.7 .6
\put{$\ssize r$} at 3.0 1.1
\multiput{$\scriptstyle\bullet$} at 1.67 2  2 2.5  2.5 .75  2.84 1.25  .56 1.25  .26 1.7 /
\multiput{$\ssize 1$} at 1.4 2  2.84 1.55 .66 1.5 /
\multiput{$\ssize \text-1$} at 2.3 2.5  2.2 .75  .46 1.95  /
\endpicture
$
&
$
\beginpicture\setcoordinatesystem units <.6cm,.6cm>
\multiput{} at 0 0  4.5 4.5 /
\put{ } at -3 2
\plot -.3 3.9  3.5 3.9 /
\plot -.3 3.45  3.5 3.45 /
\plot 1.7 .45  0 3  3.5 3 /
\plot 2 0  4.3 3.45 /
\plot 1.4 0  -.9 3.45 /
\setdots<2pt>
\plot .86 .8  2.33 3 /
\plot 2.5 .75  1 3 /
\plot 2.84 1.25  1.67 3 /
\put{$\ssize s$} at 0 1.55
\put{$\ssize r$} at .3 1.1
\put{$\ssize m$} at .6 .65
\put{$\ssize s$} at 2.7 .6
\put{$\ssize r$} at 3.0 1.1
\multiput{$\scriptstyle\bullet$} at 1.67 2  2 2.5  2.5 .75  2.84 1.25  .56 1.25  .26 1.7 /
%\multiput{$\ssize 1$} at 1.4 2  2.84 1.55 .66 1.5 /
%\multiput{$\ssize \text-1$} at 2.3 2.5  2.2 .75  .46 1.95  /
\setshadegrid span <.3mm>
\vshade 1.87 2 2 <,z,,>
        2.2 1.5 2.5  <z,z,,>
        2.6  .9  1.9  <z,z,,>
        2.94 1.4 1.4 /
\vshade .36 1.55 1.55 <,z,,>
        .46 1.4 1.7 <z,z,,>
        .66 1.1 1.4 <z,z,,>
        .76 1.25 1.25 /
\endpicture
$\\
\cline{2-3}
&\multicolumn{2}{|c|}{$0\longrightarrow B_2^{m,s}\longrightarrow B_2^{m,r}\oplus P_1^s\longrightarrow P_1^r\longrightarrow 0$}
\\
\hline
\hline
\end{tabular}
\newpage
\renewcommand{\arraystretch}{2}
\begin{tabular}{|r|@{} c | c |}
\hline
\multicolumn{3}{|c|}{\bf Matrices for moves {\bf(A)}-{\bf(D)}}\\
\hline
\multirow{4}{*}{{\bf(C)}} &
\multicolumn{2}{|c|}{$\begin{picture}(20,8)(0,3)
\put(0,4){\line(1,0){20}}
\multiput(3,3)(4,0)4{$\bullet$}
\put(6,4){\oval(4,4)[t]}
\put(14,4){\oval(4,4)[t]}
\put(4,1){\makebox[0mm]{$\scriptstyle m$}}
\put(8,1){\makebox[0mm]{$\scriptstyle n$}}
\put(12,1){\makebox[0mm]{$\scriptstyle r$}}
\put(16,1){\makebox[0mm]{$\scriptstyle s$}}
\end{picture}
\;
\leq_{\sf arc}
\;
\begin{picture}(20,8)(0,3)
\put(0,4){\line(1,0){20}}
\multiput(3,3)(4,0)4{$\bullet$}
\put(8,4){\oval(8,8)[t]}
\put(12,4){\oval(8,8)[t]}
\put(4,1){\makebox[0mm]{$\scriptstyle m$}}
\put(8,1){\makebox[0mm]{$\scriptstyle n$}}
\put(12,1){\makebox[0mm]{$\scriptstyle r$}}
\put(16,1){\makebox[0mm]{$\scriptstyle s$}}
\end{picture}
$
where $m>n>r>s$.}
\\
\cline{2-3}&$\delta M(Y,Z)$&$\delta H(Y,Z)$
\\
\cline{2-3}
&$
\beginpicture\setcoordinatesystem units <.6cm,.6cm>
\multiput{} at 0 0  4.5 4.5 /
\put{ } at -3 2
\plot -.3 3.9  3.5 3.9 /
\plot -.3 3.45  3.5 3.45 /
\plot 1.7 .45  0 3  3.5 3 /
\plot 2 0  4.3 3.45 /
\plot 1.4 0  -.9 3.45 /
\setdots<2pt>
\plot .43 1.45  1.47 3 /
\plot .77 .95  2.13 3 /
\plot 1.1 .45 2.8 3 /
\plot 2.5 .75  1 3 /
\plot 2.83 1.25  1.67 3 /
\plot 3.17 1.75  2.33 3 /
\put{$\ssize r$} at .1 1.3
\put{$\ssize n$} at .5 .8
\put{$\ssize m$} at .9 .3
\put{$\ssize s$} at 2.7 .6
\put{$\ssize r$} at 3 1.1
\put{$\ssize n$} at 3.3 1.6
\multiput{$\scriptstyle\bullet$} at 1.57 2.15  2.23 2.15  1.23 2.65  2.55 2.65 /
\multiput{$\ssize 1$} at 1.37 2.15  2.43 2.15 /
\multiput{$\ssize \text-1$} at .93 2.65  2.85 2.65 /
\endpicture
$
&
$
\beginpicture\setcoordinatesystem units <.6cm,.6cm>
\multiput{} at 0 0  4.5 4.5 /
\put{ } at -3 2
\plot -.3 3.9  3.5 3.9 /
\plot -.3 3.45  3.5 3.45 /
\plot 1.7 .45  0 3  3.5 3 /
\plot 2 0  4.3 3.45 /
\plot 1.4 0  -.9 3.45 /
\setdots<2pt>
\plot .43 1.45  1.47 3 /
\plot .77 .95  2.13 3 /
\plot 1.1 .45 2.8 3 /
\plot 2.5 .75  1 3 /
\plot 2.83 1.25  1.67 3 /
\plot 3.17 1.75  2.33 3 /
\put{$\ssize r$} at .1 1.3
\put{$\ssize n$} at .5 .8
\put{$\ssize m$} at .9 .3
\put{$\ssize s$} at 2.7 .6
\put{$\ssize r$} at 3 1.1
\put{$\ssize n$} at 3.3 1.6
\multiput{$\scriptstyle\bullet$} at 1.57 2.15  2.23 2.15  1.23 2.65  2.55 2.65 /
%\multiput{$\ssize 1$} at 1.37 2.15  2.43 2.15 /
%\multiput{$\ssize \text-1$} at .93 2.65  2.85 2.65 /
\setshadegrid span <.3mm>
\vshade .53 1.3 1.3 <,z,,>
        .87 .8  1.8  <z,z,,>
        1.43  1.65 2.65 <z,z,,>
        1.77 2.15 2.15 /
\vshade 2.43 2.15 2.15 <,z,,>
        2.75 1.65 2.65 <z,z,,>
        2.93 1.4 2.4  <z,z,,>
        3.27 1.9 1.9  /
\endpicture
$
\\
\cline{2-3}
&\multicolumn{2}{|c|}{$0\longrightarrow B_2^{m,r}\longrightarrow B_2^{m,n}\oplus B_2^{r,s}\longrightarrow B_2^{n,s}\longrightarrow 0$}
\\
\hline
\hline
\multirow{4}{*}{{\bf(D)}} &
\multicolumn{2}{|c|}{$\begin{picture}(16,8)(0,3)
\put(0,4){\line(1,0){16}}
\multiput(3,3)(4,0)3{$\bullet$}
\put(10,4){\oval(4,4)[t]}
\put(4,4){\line(0,1){7}}
\put(4,1){\makebox[0mm]{$\scriptstyle m$}}
\put(8,1){\makebox[0mm]{$\scriptstyle r$}}
\put(12,1){\makebox[0mm]{$\scriptstyle s$}}
\end{picture}
\;
\leq_{\sf arc}
\;
\begin{picture}(16,8)(0,3)
\put(0,4){\line(1,0){16}}
\multiput(3,3)(4,0)3{$\bullet$}
\put(8,4){\oval(8,8)[t]}
\put(8,4){\line(0,1){7}}
\put(4,1){\makebox[0mm]{$\scriptstyle m$}}
\put(8,1){\makebox[0mm]{$\scriptstyle r$}}
\put(12,1){\makebox[0mm]{$\scriptstyle s$}}
\end{picture}
$
where $m>r>s$.}
\\
\cline{2-3}&$\delta M(Y,Z)$&$\delta H(Y,Z)$
\\
\cline{2-3}
&$
\beginpicture\setcoordinatesystem units <.6cm,.6cm>
\multiput{} at 0 0  4.5 4.5 /
\put{ } at -3 2
\plot -.3 3.9  3.5 3.9 /
\plot -.3 3.45  3.5 3.45 /
\plot 1.7 .45  0 3  3.5 3 /
\plot 2 0  4.3 3.45 /
\plot 1.4 0  -.9 3.45 /
\setdots<2pt>
\plot .43 1.45  1.47 3 /
\plot .77 .95  2.13 3 /
%\plot 1.1 .45 2.8 3 /
\plot 2.5 .75  1 3 /
\plot 2.83 1.25  1.67 3 /
\plot 3.17 1.75  2.33 3 /
\put{$\ssize r$} at .1 1.3
\put{$\ssize m$} at .5 .8
%\put{$\ssize m$} at .9 .3
\put{$\ssize s$} at 2.7 .6
\put{$\ssize r$} at 3 1.1
\put{$\ssize m$} at 3.3 1.6
\multiput{$\scriptstyle\bullet$} at 1.57 2.15  1.23 2.65  .43 1.45  .77 .95  2.83 1.25  3.17 1.75 /
\multiput{$\ssize 1$} at 1.77 2.15  .43 1.75  2.63 1.25  /
\multiput{$\ssize \text-1$} at .93 2.65  1.07 .95  2.87 1.75 /
\endpicture
$
&
$
\beginpicture\setcoordinatesystem units <.6cm,.6cm>
\multiput{} at 0 0  4.5 4.5 /
\put{ } at -3 2
\plot -.3 3.9  3.5 3.9 /
\plot -.3 3.45  3.5 3.45 /
\plot 1.7 .45  0 3  3.5 3 /
\plot 2 0  4.3 3.45 /
\plot 1.4 0  -.9 3.45 /
\setdots<2pt>
\plot .43 1.45  1.47 3 /
\plot .77 .95  2.13 3 /
%\plot 1.1 .45 2.8 3 /
\plot 2.5 .75  1 3 /
\plot 2.83 1.25  1.67 3 /
\plot 3.17 1.75  2.33 3 /
\put{$\ssize r$} at .1 1.3
\put{$\ssize m$} at .5 .8
%\put{$\ssize m$} at .9 .3
\put{$\ssize s$} at 2.7 .6
\put{$\ssize r$} at 3 1.1
\put{$\ssize m$} at 3.3 1.6
\multiput{$\scriptstyle\bullet$} at 1.57 2.15  1.23 2.65  .43 1.45  .77 .95  2.83 1.25  3.17 1.75 /
%\multiput{$\ssize 1$} at 1.77 2.15  .43 1.75  2.63 1.25  /
%\multiput{$\ssize \text-1$} at .93 2.65  1.07 .95  2.87 1.75 /
\setshadegrid span <.3mm>
\vshade .53 1.3 1.3 <,z,,>
        .87 .8  1.8  <z,z,,>
        1.43  1.65 2.65 <z,z,,>
        1.77 2.15 2.15 /
\vshade 2.83 1.55 1.55 <,z,,>
        2.93 1.4 1.7   <z,z,,>
        3.17 1.75 2.05 <z,z,,>
        3.27 1.9 1.9 /
\endpicture
$
\\
\cline{2-3}
&\multicolumn{2}{|c|}{$0\longrightarrow P_1^r\longrightarrow P_1^m\oplus B_2^{r,s}\longrightarrow B_2^{m,s}\longrightarrow 0$}
\\
\hline
\hline
\end{tabular}
\end{center}

This completes the proof of the following fact.

\begin{lem}\label{lem-tab-ext}
  Let $k$ be an~arbitrary field and let
  $Y,Z\in\mathcal{S}_2(k)$ have the same partition type
  $(\beta,\gamma)$.
  If $Y\leq_{\sf arc} Z$, then $Y\leq_{\sf ext} Z$.\qed
\end{lem}

Note that the Hom matrix determines the multiplicity matrix uniquely.
Namely, let $A$ be a non-injective indecomposable object with
Auslander-Reiten sequence $0\to A\to \bigoplus B_i\to C\to 0$
in $\mathcal S_2^n$.  Then the multiplicity of $A$ as a direct summand
of $Y$ is given by the contravariant defect
$\mu_A(Y)=[A,Y]+[C,Y]-\sum_i [B_i,Y]$.

\section{Some technical facts}\label{section-technical}
%==========================================================

We recall the following result from \cite[Lemma~4.2]{kos-sch}.

\begin{lem}\label{lemma-zeros}
Suppose $Y,Z\in\mathcal S_2$ have the same partition type $(\beta,\gamma)$.
\begin{enumerate}
\item The Hom matrix $\delta H(Y,Z)$ has zero entry at each position corresponding
  to a module $P_0^m$, $P_2^m$ where $m\in \mathbb N$.
\item Along each diagonal in the Hom matrix, the entries eventually become constant:
  $$\lim_{m\to\infty}\delta H(Y,Z)_{B_2^{m,r}}\;=\; \delta H(Y,Z)_{P_1^m}$$
\end{enumerate}
\end{lem}

%\smallskip
The following consequence, which is technical but easy to show, will be used
in the proof of Proposition~\ref{proposition-key-lemma}.

  \begin{lem}\label{lem-hom-yields-mult}
    Suppose that both $Y,Z\in\mathcal S_2^n$ have partition type $(\beta,\gamma)$.
    \begin{enumerate}
    \item The nonzero part of the Hom matrix $\delta H(Y,Z)$
      is contained in the union of the $\tau$-orbits for $X=B_2^{3,1},\ldots, B_2^{n,1}$;
      they form a Moebius band, i.e.\ a quiver of type $\mathbb Z \mathbb A_{n-2}$
      with suitable identifications.
    \item
      For each non-injective $A\in\ind\,\mathcal S_2^n$, the entry $\delta M_A=\delta M(Y,Z)_A$
      in the multiplicity matrix
      can be read off from the the restriction of the Hom matrix to the Moebius band
      by a formula of type
      $$\delta M_A=\beta_A+\beta_C-\sum_i\beta_{B_i}$$
      where $0\to A\to \bigoplus_i B_i\to C\to 0$ is the Auslander-Reiten sequence starting at $A$ and
      $$\beta_X=\left\{\begin{array}{ll}\delta H(Y,Z)_X & \text{if $X$ is in the Moebius band}\\
      0               & \text{otherwise.}\end{array}\right.$$
    \end{enumerate}
  \end{lem}

  \begin{proof}
    The first statement follows from Lemma~\ref{lemma-zeros}, the second
    from the contravariant defect formula above.
  \end{proof}

In the proof we will use \cite[Lemmata~4.6 and 4.7]{kos-sch}.

\begin{lem}\label{lemma-positive-region}
Consider the following matrix of integers.
$$
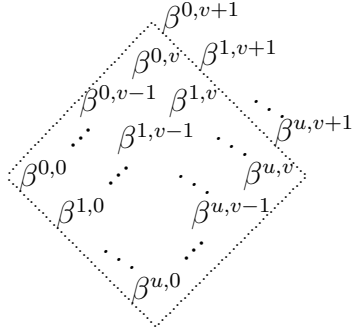
\beginpicture\setcoordinatesystem units <.5cm,.5cm>
\multiput{} at 0 0  7 7 /
\put{$\beta^{0,0}$} at 0 3
\put{$\beta^{1,0}$} at 1 2
\put{$\beta^{u,0}$} at 3 0
\put{$\beta^{0,v-1}$} at 2 5
\put{$\beta^{1,v-1}$} at 3 4
\put{$\beta^{u,v-1}$} at 5 2
\put{$\beta^{0,v}$} at 3 6
\put{$\beta^{1,v}$} at 4 5
\put{$\beta^{u,v}$} at 6 3
\put{$\beta^{0,v+1}$} at 4.2 7.2
\put{$\beta^{1,v+1}$} at 5.2 6.2
\put{$\beta^{u,v+1}$} at 7.2 4.2
\multiput{$\cdot$} at .8 3.8  1 4  1.2 4.2
                      1.8 2.8  2 3  2.2 3.2
                      3.8 .8  4 1  4.2 1.2 /
\multiput{$\ddots$} at 2 1  4 3  5 4  6 5 /
\setdots<2pt>
\plot -.9 3.1  2.9 7.1  7 3  3 -1  -1 3 /
\endpicture$$
Suppose the following conditions are satisfied.
\begin{enumerate}
\item All entries are nonnegative
\item The numbers $\beta^{0,0}, \beta^{0,1},\ldots,\beta^{0,v}$ are
      strictly positive
\item $\beta^{0,v+1}=0$
\item For each $0\leq i\leq u-1$, $0\leq j\leq v$,
$$\delta^{ij}=\beta^{i,j}+\beta^{i+1,j+1}-\beta^{i,j+1}-\beta^{i+1,j}\leq 0$$
\end{enumerate}
Then all entries in the parallelogram are positive:
$\beta^{i,j}>0$ for each $0\leq i\leq u$, $0\leq j\leq v$.
\end{lem}

Also the dual version holds:

\begin{lem}\label{lemma-positive-region-ii}
Consider the following matrix of integers.
$$
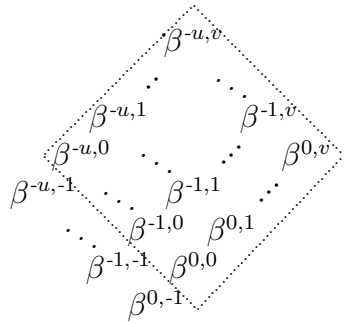
\beginpicture\setcoordinatesystem units <.5cm,.5cm>
\multiput{} at 0 0  7 7 /
\put{$\beta^{0,v}$} at 7 4
\put{$\beta^{\text-1,v}$} at 6 5
\put{$\beta^{\text-u,v}$} at 4 7
\put{$\beta^{0,1}$} at 5 2
\put{$\beta^{\text-1,1}$} at 4 3
\put{$\beta^{\text-u,1}$} at 2 5
\put{$\beta^{0,0}$} at 4 1
\put{$\beta^{\text-1,0}$} at 3 2
\put{$\beta^{\text-u,0}$} at 1 4
\put{$\beta^{0,\text-1}$} at 3 0
\put{$\beta^{\text-1,\text-1}$} at 2 1
\put{$\beta^{\text-u,\text-1}$} at 0 3
\multiput{$\cdot$} at 6.2 3.2  6 3  5.8 2.8  5.2 4.2  5 4  4.8 3.8  3.2 7.2  3 6  2.8 5.8 /
\multiput{$\ddots$} at 5 6  3 4  2 3  1 2 /
\setdots<2pt>
\plot 7.9 3.9  4.1 -.1  0 4  4 8  8 4 /
\endpicture$$
Suppose that in addition to conditions 1 and 2 from the previous lemma also the following are satisfied.
\begin{enumerate}
\item[3'.] $\beta^{0,-1}=0$
\item[4'.] For each $-u\leq i<0$, $-1\leq j<v$,
$$\delta^{ij}=\beta^{i,j}+\beta^{i+1,j+1}-\beta^{i,j+1}-\beta^{i+1,j}\leq 0$$
\end{enumerate}
Then all entries in the parallelogram are positive:
$\beta^{i,j}>0$ for each $-u\leq i\leq 0$, $0\leq j\leq v$.
\qed
\end{lem}

\section{The hom-order implies the arc-order}
%=============================================================
\label{chapter-hom-tab}

Our aim is to show
that $Y \leq_{\sf hom} Z$ implies $Y\leq_{\sf arc} Z$.

\begin{thm}\label{theorem-hom-arc}
  Suppose the objects $Y$ and $Z$ in $\mathcal S_2(k)$ have
  the same partition type $(\beta,\gamma)$.
    If $Y \leq_{\sf hom} Z$ holds, then so does $Y\leq_{\sf arc} Z$.
\end{thm}

This theorem follows from the following.

\begin{prop}\label{proposition-key-lemma}
Suppose $Y$ and $Z$ have the same
partition type $(\beta,\gamma)$, $Y\leq_{\sf hom}Z$ and
$Y\not\cong Z$.  Then there is an operation on the arc diagram for $Z$
of type {\bf (A)}, {\bf (B)}, {\bf (C)}, {\bf (D)} or {\bf (E)}
which yields a module $Z'$ such that
$$ Z'\leq_{\sf arc} Z,\qquad Z'\not\cong Z,\qquad\text{and}\qquad
   Y\leq_{\sf hom}Z'.$$
\end{prop}

\begin{proof}
The methods from the proof of \cite[Proposition~4.5]{kos-sch} can be adapted.
Here, we have
to consider one more case, namely the move of the type {\bf (E)}. For the convenience
of the reader we present the proof with all details. Moreover, after the proof we will 
illustrate it by an example.

\subsection{The proof}
%---------------------

%\smallskip
{\bf The set-up.} We assume that the entries in the Hom-matrix
$\delta H(Y,Z)$ are all nonnegative and that at least one entry is positive.

%\smallskip
{\bf The goal.} We show that there is a parallelogram in the
Hom-matrix (in the shape of one of the shaded
regions in Section~\ref{section-operations})
which satisfies the following two conditions.
\begin{itemize}
\item[(P1)] All entries within the parallelogram are strictly positive.
\item[(P2)] The two indecomposable modules $X'$ and $X''$ corresponding to the
  right corner of the parallelogram and to the point just left of the
  left corner, respectively, occur with higher multiplicity as direct summands of $Z$
  than as a direct summand of $Y$.
\end{itemize}

{\bf Step 1.}
We choose a number $n>\beta_1$ and work in the category $\mathcal S_2^n$.
Recall that in the Hom matrix, the nonzero entries occur in the
orbits given by the modules $B_2^{3,1},\ldots,B_2^{n,1}$.  The region given by those orbits 
forms a stripe of
type $\mathbb Z\mathbb A_{n-2}$, with identifications (see Lemma \ref{lemma-zeros}).
For the purpose of this algorithm, we view the Hom matrix as a stripe of
type $\mathbb Z\mathbb A_n$, with the orbits of the $B_2^{3,1},\ldots,B_2^{n,1}$
in the center and two orbits of zeros on the boundary.

%\smallskip
Pick a sequence $\beta^{0,0},\ldots,\beta^{0,v}$ of positive entries
arranged along an anti-diagonal in the Hom matrix, such that the neighboring entries
on the anti-diagonal, $\beta^{0,\text-1}$ and $\beta^{0,v+1}$, are both zero.
This is possible since the entries in the matrix are nonnegative, since there
is at least one positive entry, and since there are zeros on the boundary orbits (see Lemma \ref{lemma-zeros}).

%\smallskip
% In the example, we have picked the anti-diagonal
% $\beta^{0,0}=\delta H_{B_2^{6,2}}=1,\beta^{0,1}=\delta H_{B_2^{6,3}}=1,\beta^{0,2}=\delta H_{B_2^{6,4}}=1$.

%\smallskip
{\bf Step 2.}
We determine the right corner of the parallelogram.  Put $u=0$.

%\smallskip
Consider the sequence $\delta^{u,0},\ldots,\delta^{u,v}$
of entries in the multiplicity matrix in the positions given by
the $\beta^{u,0},\ldots,\beta^{u,v}$ in the Hom matrix.

%\smallskip
If one of the entries in the sequence $\delta^{u,0},\ldots,\delta^{u,v}$
is positive, put $u''=u$ and let $\delta^{u'',w''}$ be the first such entry.
In this case, $(u'',w'')$ will be the right corner of the rectangle and
$X''$, the object corresponding to the position $(u'',w'')$, will be a summand
occurring with higher multiplicity in $Z$ than in $Y$.  We are done with the second step.

%\smallskip
If none of the entries in $\delta^{u,0},\ldots,\delta^{u,v}$ is positive,
then by  Lemma \ref{lem-hom-yields-mult} we obtain that the assumptions of Lemma~\ref{lemma-positive-region} are satisfied, 
so using this lemma we obtain that all the numbers
$\beta^{u+1,0},\ldots,\beta^{u+1,v}$ in the Hom matrix, on the anti-diagonal just
under the previous anti-diagonal, are positive.
Put $u:=u+1$ and proceed with the paragraph under Step~2.

%\smallskip
Note that this process terminates:  For $u$ large enough, $\beta^{u,0}$
will correspond to a point on the boundary of the Hom matrix.  Hence $\beta^{u,0}=0$.

%\smallskip
% In the example, the number $\delta^{0,2}=\delta M_{B_2^{6,3}}=\beta^{0,2}+\beta^{1,3}-\beta^{0,3}-\beta^{1,2}=1$
% is the first positive value among the entries in the multiplicity matrix on the
% anti-diagonal.  Thus, $X''=B_2^{6,3}$.

%\smallskip
{\bf Step 3.}
We determine the left corner of the parallelogram.  Put $u=0$ and $v=v''$.

%\smallskip
Consider the sequence $\delta^{u-1,-1},\ldots,\delta^{u-1,v-1}$
of entries in the multiplicity matrix in the positions just left of
the $\beta^{u,0},\ldots,\beta^{u,v}$ in the Hom matrix.

%\smallskip
If one of the entries in the sequence $\delta^{u-1,-1},\ldots,\delta^{u-1,v-1}$
is positive, put $u'=u-1$ and let $\delta^{u',w'}$ be the last such entry.
In this case, $(u',w')$ will be the point just left of the left corner of the rectangle and
$X'$, the object corresponding to the position $(u',w')$, will be a summand
occurring with higher multiplicity in $Z$ than in $Y$.  We are done with the third step.

%\smallskip
If none of the entries in $\delta^{u-1,-1},\ldots,\delta^{u-1,v-1}$ is positive,
then by Lemma \ref{lem-hom-yields-mult} we obtain that the assumptions of Lemma~\ref{lemma-positive-region-ii} 
are satisfied, so using this lemma we obtain that all the numbers
$\beta^{u-1,0},\ldots,\beta^{u-1,v}$ in the Hom matrix, on the anti-diagonal just
above the previous anti-diagonal, are positive.  %{\red wyjasnic dlaczego jest tak jak w poprzednim zdaniu. Z czego to wynika?
%To samo dopisac w Step 2.}
Put $u:=u-1$ and proceed with the paragraph under Step~3.

%\smallskip
Note that this process terminates: For $-u$ large enough, the position corresponding to
$\beta^{u,0}$ will be a point on the boundary of the Hom matrix, so $\beta^{u,0}=0$.

%\smallskip
% In the example, $\delta^{-1,-1}=\delta M_{B_2^{5,1}}=1$ is the last (and first) positive
% entry in the multiplicity matrix on the line above the original anti-diagonal.
% Thus, $X'=B_2^{5,1}$.

%\smallskip
{\bf Conclusion.}
In each case the parallelogram marked off by $X'$ and $X''$
is of one of the types {\bf (A)}, {\bf (B)}, {\bf (C)}, {\bf (D)}  or  {\bf (E)} in Section~\ref{section-operations}.
(Namely, if none of the modules $P_1^r$ is in the rectangle, then the type is {\bf (A)};
if $X''$ is one of the $P_1^r$, then the type is {\bf (B)}; in case $X'\cong P_1^r$ for some $r$,
then the type is {\bf (D)}; and if there are modules $P_1^r$ (for $r\neq 1$) in the rectangle,
but neither $X'$ nor $X''$ has this form, then the type is {\bf (C)}; finally if the modul $P_1^1$ is in the rectangle,
then the type is {\bf (E)}.)
Write $Z=Z_0\oplus X'\oplus X''$  and let $X_1,\ldots,X_s$ ($s=2$ or $3$)
be the  modules corresponding to the entry $-1$ in the multiplicity matrix
for the arc operation.
Put $Z'=Z_0\oplus X_1\oplus\cdots\oplus X_s$.  By this construction,
$Z'\leq_{\sf arc} Z$. We obtain the arc diagram of $Z'$ from the arc diagram of $Z$
by a single move of type ${\bf (A)},\,{\bf (B)},\,{\bf (C)},\,{\bf (D)}$ or ${\bf (E)}$, 
so the hom matrix was changed in the way described in Section \ref{section-operations}, so we obtain that
$$\delta H(Y,Z')_X\;=\;\delta H(Y,Z)_X+\left\{\begin{array}{ll}-1 &
                            \text{if $X$ is in the parallelogram}\\
                     0 & \text{otherwise}\end{array}\right.$$
Since the entries for $\delta H(Y,Z)$ within the parallelogram are all
positive, the matrix $\delta H(Y,Z')$ is nonnegative and hence
$$Y\leq_{\sf hom} Z'\qquad\text{and}\qquad Z'\leq_{\sf arc}Z.$$
\end{proof}

\begin{proof}[Proof of Theorem~\ref{theorem-hom-arc}]
Assume that $Y\leq_{\sf hom}Z$.
Since each application of Proposition~\ref{proposition-key-lemma}
reduces the number of crossings in the arc diagram,
a finite number of steps
suffices to produce a sequence of modules
$Z$, $Z'$, $Z''$, $\ldots$, $Z^{(m)}=Y$ such that
$$Y=Z^{(m)}\leq_{\sf arc} Z^{(m-1)}\leq_{\sf arc} \cdots\leq_{\sf arc} Z''
          \leq_{\sf arc} Z' \leq_{\sf arc} Z.$$
\end{proof}

\subsection{An example}\label{section-example}
%----------------------

In this section we present an example illustrating the proof of Proposition \ref{proposition-key-lemma}.  The arc diagrams
$$
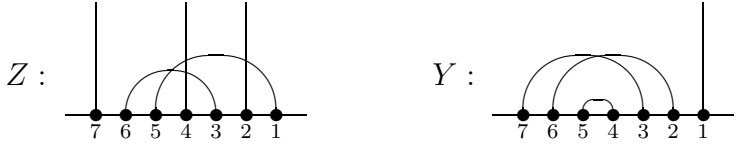
\begin{picture}(32,15)(0,3)
\put(-8,8){$Z:$}
\put(0,4){\line(1,0){32}}
\multiput(3,3)(4,0)7{$\bullet$}
\put(20,4){\oval(16,16)[t]}
\put(14,4){\oval(12,12)[t]}
\put(4,4){\line(0,1){15}}
\put(16,4){\line(0,1){15}}
\put(24,4){\line(0,1){15}}
\put(4,1){\makebox[0mm]{$\scriptstyle 7$}}
\put(8,1){\makebox[0mm]{$\scriptstyle 6$}}
\put(12,1){\makebox[0mm]{$\scriptstyle 5$}}
\put(16,1){\makebox[0mm]{$\scriptstyle 4$}}
\put(20,1){\makebox[0mm]{$\scriptstyle 3$}}
\put(24,1){\makebox[0mm]{$\scriptstyle 2$}}
\put(28,1){\makebox[0mm]{$\scriptstyle 1$}}
\end{picture}
\qquad\qquad\qquad
\begin{picture}(32,15)(0,3)
\put(-8,8){$Y:$}
\put(0,4){\line(1,0){32}}
\multiput(3,3)(4,0)7{$\bullet$}
\put(12,4){\oval(16,16)[t]}
\put(16,4){\oval(16,16)[t]}
\put(14,4){\oval(4,4)[t]}
\put(28,4){\line(0,1){15}}
\put(4,1){\makebox[0mm]{$\scriptstyle 7$}}
\put(8,1){\makebox[0mm]{$\scriptstyle 6$}}
\put(12,1){\makebox[0mm]{$\scriptstyle 5$}}
\put(16,1){\makebox[0mm]{$\scriptstyle 4$}}
\put(20,1){\makebox[0mm]{$\scriptstyle 3$}}
\put(24,1){\makebox[0mm]{$\scriptstyle 2$}}
\put(28,1){\makebox[0mm]{$\scriptstyle 1$}}
\end{picture}
$$
represent the modules
$Y=B_2^{7,3}\oplus B_2^{6,2}\oplus P_2^5\oplus P_0^4\oplus P_1^1$
and $Z= B_2^{6,3}\oplus B_2^{5,1}\oplus P_1^{7}\oplus  P_1^4 \oplus  P_1^2$.

%\medskip
We first compute the Hom-matrix to verify that $Y\leq_{\sf hom}Z$.
In this matrix, the entries on the first diagonal represent the numbers
$\delta H_{B_2^{3,1}},\delta H_{B_2^{4,1}},\ldots$.
By Lemma~\ref{lemma-zeros}, there are also zeros along the
two top rows, which we indicate by solid lines.

$$
\beginpicture\setcoordinatesystem units <.4cm,.4cm>
\multiput{} at 0 0  4 4 /
\put{$\delta H(Y,Z):$} at -3 2
\plot -.3 3.9  6.5 3.9 /
\plot -.3 4.1  6.5 4.1  /
\multiput{$\scriptstyle 0$} at 0 3.5  1 3.5  2 3.5  4 3.5  5 3.5
                               .5 3  1.5 3  4.5 3  6.5 3
                               1 2.5  3 2.5  6 2.5
                               1.5 2  5.5 2
                               2 1.5  /
\multiput{$\scriptstyle 1$} at 3 3.5
                               2.5 3  3.5 3
                               2 2.5  4 2.5
                               2.5 2  3.5 2
                               5 1.5
                               2.5 1  4.5 1
                               3.5 0  /
\multiput{$\scriptstyle 2$} at	3 1.5 
                                 4 .5 /
\multiput{$\cdot$} at        2.8 .7  3 .5  3.2 .3
                             3.3 1.2  3.5 1  3.7 .8
                             3.8 1.7  4 1.5  4.2 1.3
                             4.3 2.2  4.5 2  4.7 1.8
                             4.8 2.7  5 2.5  5.2 2.3
                             6.8 3.3  7 3.5  7.2 3.7 /
\endpicture
$$

We see that $Y\leq_{\sf hom}Z$. We show that $Y\leq_{\sf arc}Z$.

Following Steps 1,2,3 in the proof of Proposition \ref{proposition-key-lemma},  we choose the smaller 
parallelogram in the Hom-matrix, as indicated.

$$
\beginpicture\setcoordinatesystem units <.4cm,.4cm>
\multiput{} at 0 0  4 4 /
\put{$\delta H(Y,Z):$} at -3 2
\plot -.3 3.9  6.5 3.9 /
\plot -.3 4.1  6.5 4.1  /
\multiput{$\scriptstyle 0$} at 0 3.5  1 3.5  2 3.5  4 3.5  5 3.5
                               .5 3  1.5 3  4.5 3  6.5 3
                               1 2.5  3 2.5  6 2.5
                               1.5 2  5.5 2
                               2 1.5 /
\multiput{$\scriptstyle 1$} at 3 3.5
                               2.5 3  3.5 3
                               2 2.5  4 2.5
                               2.5 2  3.5 2
                               5 1.5
                               2.5 1  4.5 1
                               3.5 0  /
\multiput{$\scriptstyle 2$} at	3 1.5 
                                 4 .5 /
\multiput{$\cdot$} at        2.8 .7  3 .5  3.2 .3
                             3.3 1.2  3.5 1  3.7 .8
                             3.8 1.7  4 1.5  4.2 1.3
                             4.3 2.2  4.5 2  4.7 1.8
                             4.8 2.7  5 2.5  5.2 2.3
                             6.8 3.3  7 3.5  7.2 3.7 /
\setdots<2pt>
\plot 2.5 3.5  1.5 2.5  2 2  3 3  2.5 3.5 /
\endpicture
$$

% There are two reasons for considering this parallelogram:
% (1) On the left of the left corner of the parallelogram, and in the right corner,
% there are positive entries in the multiplicity
% diagram, see Section~\ref{section-operations}.  Those entries
% indicate the arcs and poles involved in the operation which
% produces $Z'$ from $Z$.
% 
% %\smallskip
% (2) Unlike the big parallelogram, there are only positive
% entries in the smaller parallelogram.  This will make sure that
% $Y\leq_{\sf hom} Z'$ holds.

%\medskip
The corresponding operation of type {\bf (A)}
replaces the bipickets $B_2^{5,1}$ and $B_2^{6,3}$ in $Z$ (representing arcs
from 5 to 1 and from 6 to 3, respectively)
by bipickets $B_2^{6,1}$ and $B_2^{5,3}$ in $Z'$, so $Z'$ is given by
the following arc diagram.  We also indicate the new Hom-matrix.

$$
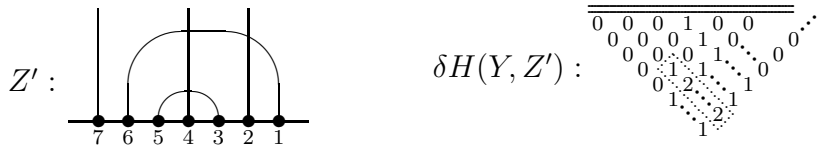
\begin{picture}(32,15)(0,3)
\put(-8,8){$Z':$}
\put(0,4){\line(1,0){32}}
\multiput(3,3)(4,0)7{$\bullet$}
\put(18,6){\oval(20,20)[t]}
\put(16,4){\oval(8,8)[t]}
\put(4,4){\line(0,1){15}}
\put(16,4){\line(0,1){15}}
\put(24,4){\line(0,1){15}}
\put(8,4){\line(0,1){2}}
\put(28,4){\line(0,1){2}}
\put(4,1){\makebox[0mm]{$\scriptstyle 7$}}
\put(8,1){\makebox[0mm]{$\scriptstyle 6$}}
\put(12,1){\makebox[0mm]{$\scriptstyle 5$}}
\put(16,1){\makebox[0mm]{$\scriptstyle 4$}}
\put(20,1){\makebox[0mm]{$\scriptstyle 3$}}
\put(24,1){\makebox[0mm]{$\scriptstyle 2$}}
\put(28,1){\makebox[0mm]{$\scriptstyle 1$}}
\end{picture}
\qquad\qquad
\beginpicture\setcoordinatesystem units <.4cm,.4cm>
\multiput{} at 0 0  4 4 /
\put{$\delta H(Y,Z'):$} at -3 2
\plot -.3 3.9  6.5 3.9 /
\plot -.3 4.1  6.5 4.1  /
\multiput{$\scriptstyle 0$} at 0 3.5  1 3.5  2 3.5  4 3.5  5 3.5
                                  .5 3  1.5 3  4.5 3  6.5 3         2.5 3
                                       1 2.5  3 2.5  6 2.5          2 2.5
                                           1.5 2  5.5 2
                                             2 1.5  /
\multiput{$\scriptstyle 1$} at 3 3.5
                                      3.5 3
                                      4 2.5
                               2.5 2  3.5 2
                               5 1.5
                               2.5 1  4.5 1
                               3.5 0  /
\multiput{$\scriptstyle 2$} at	3 1.5 
                                 4 .5 /
\multiput{$\cdot$} at        2.8 .7  3 .5  3.2 .3
                             3.3 1.2  3.5 1  3.7 .8
                             3.8 1.7  4 1.5  4.2 1.3
                             4.3 2.2  4.5 2  4.7 1.8
                             4.8 2.7  5 2.5  5.2 2.3
                             6.8 3.3  7 3.5  7.2 3.7 /
\setdots<2pt>
\plot   2.5 2.5  2 2  4 0  4.5 0.5  2.5 2.5   /
\endpicture
$$

Following Steps 1,2,3 in the proof of Proposition \ref{proposition-key-lemma}, we choose a~parallelogram 
(alternatively, we can choose the parallelogram
given by the upper diagonal in the region marked by the 1's)
and perform the operation indicated:  Replace in $Z'$
the bipicket $B_2^{6,1}$ and picket $P_1^{2}$ by bipicket $B_2^{6,2}$ and picket $P_1^{1}$
in $Z''$.  Arc diagram and Hom-matrix are as follows.

$$
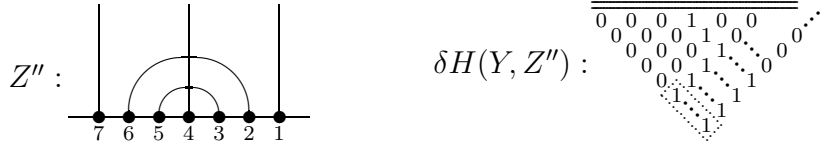
\begin{picture}(32,15)(0,3)
\put(-8,8){$Z'':$}
\put(0,4){\line(1,0){32}}
\multiput(3,3)(4,0)7{$\bullet$}
\put(16,4){\oval(16,16)[t]}
\put(16,4){\oval(8,8)[t]}
\put(4,4){\line(0,1){15}}
\put(16,4){\line(0,1){15}}
\put(28,4){\line(0,1){15}}
\put(4,1){\makebox[0mm]{$\scriptstyle 7$}}
\put(8,1){\makebox[0mm]{$\scriptstyle 6$}}
\put(12,1){\makebox[0mm]{$\scriptstyle 5$}}
\put(16,1){\makebox[0mm]{$\scriptstyle 4$}}
\put(20,1){\makebox[0mm]{$\scriptstyle 3$}}
\put(24,1){\makebox[0mm]{$\scriptstyle 2$}}
\put(28,1){\makebox[0mm]{$\scriptstyle 1$}}
\end{picture}
\qquad\qquad
\beginpicture\setcoordinatesystem units <.4cm,.4cm>
\multiput{} at 0 0  4 4 /
\put{$\delta H(Y,Z''):$} at -3 2
\plot -.3 3.9  6.5 3.9 /
\plot -.3 4.1  6.5 4.1  /
\multiput{$\scriptstyle 0$} at 0 3.5  1 3.5  2 3.5  4 3.5  5 3.5
                                  .5 3  1.5 3  4.5 3  6.5 3         2.5 3
                                       1 2.5  3 2.5  6 2.5          2 2.5
                                           1.5 2  5.5 2             2.5 2
                                             2 1.5 /
\multiput{$\scriptstyle 1$} at 3 3.5
                                      3.5 3
                                      4 2.5
                                      3.5 2
                               3 1.5  5 1.5
                                      2.5 1 4.5 1
												  4 0.5
                                      3.5 0 /   
\multiput{$\cdot$} at        2.8 .7  3 .5  3.2 .3
                             3.3 1.2  3.5 1  3.7 .8
                             3.8 1.7  4 1.5  4.2 1.3
                             4.3 2.2  4.5 2  4.7 1.8
                             4.8 2.7  5 2.5  5.2 2.3
                             6.8 3.3  7 3.5  7.2 3.7 /
\setdots<2pt>
\plot 2 1  2.5 1.5  4 0  3.5 -0.5  2 1 /
\endpicture
$$

The new parallelogram corresponds to move of a type {\bf (E)},
so we replace 
the pickets $P_1^{7}$ and $P_1^{1}$ in $Z''$ by bipicket $B_2^{7,1}$
in $Z'''$.  Arc diagram and Hom-matrix are as follows.

$$
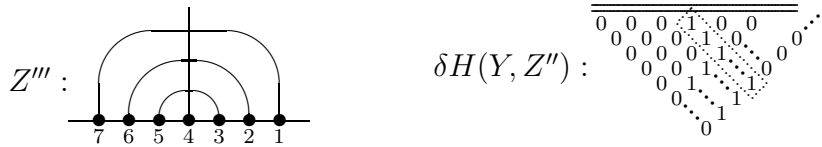
\begin{picture}(32,15)(0,3)
\put(-8,8){$Z''':$}
\put(0,4){\line(1,0){32}}
\multiput(3,3)(4,0)7{$\bullet$}
\put(16,4){\oval(24,24)[t]}
\put(16,4){\oval(16,16)[t]}
\put(16,4){\oval(8,8)[t]}
\put(16,4){\line(0,1){15}}
\put(4,1){\makebox[0mm]{$\scriptstyle 7$}}
\put(8,1){\makebox[0mm]{$\scriptstyle 6$}}
\put(12,1){\makebox[0mm]{$\scriptstyle 5$}}
\put(16,1){\makebox[0mm]{$\scriptstyle 4$}}
\put(20,1){\makebox[0mm]{$\scriptstyle 3$}}
\put(24,1){\makebox[0mm]{$\scriptstyle 2$}}
\put(28,1){\makebox[0mm]{$\scriptstyle 1$}}
\end{picture}
\qquad\qquad
\beginpicture\setcoordinatesystem units <.4cm,.4cm>
\multiput{} at 0 0  4 4 /
\put{$\delta H(Y,Z''):$} at -3 2
\plot -.3 3.9  6.5 3.9 /
\plot -.3 4.1  6.5 4.1  /
\multiput{$\scriptstyle 0$} at 0 3.5  1 3.5  2 3.5  4 3.5  5 3.5
                                  .5 3  1.5 3  4.5 3  6.5 3         2.5 3
                                       1 2.5  3 2.5  6 2.5          2 2.5
                                           1.5 2  5.5 2             2.5 2
                                             2 1.5
                                                 2.5 1
                                                    3.5 0 /
\multiput{$\scriptstyle 1$} at 3 3.5
                                      3.5 3
                                      4 2.5
                                      3.5 2
                               3 1.5  5 1.5
                                      4.5 1
                                      4 .5 /
\multiput{$\cdot$} at        2.8 .7  3 .5  3.2 .3
                             3.3 1.2  3.5 1  3.7 .8
                             3.8 1.7  4 1.5  4.2 1.3
                             4.3 2.2  4.5 2  4.7 1.8
                             4.8 2.7  5 2.5  5.2 2.3
                             6.8 3.3  7 3.5  7.2 3.7 /
\setdots<2pt>
\plot 3 4  5.5 1.5  5 1  2.5 3.5  3 4 /
\endpicture
$$

The parallelogram gives an operation of type {\bf (B)}, namly
we have to replace the bipicket $B_2^{5,3}$ and the picket $P_1^4$
in $Z'''$ by pickets $P_2^5$, $P_0^4$ and $P_1^3$ in $Z^{IV}$
(note that the first two pickets represent an arc from 5 to 4,
the last is a pole at 3).

$$
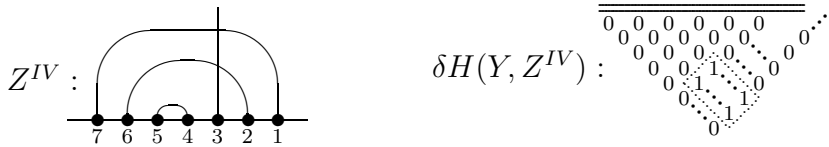
\begin{picture}(32,15)(0,3)
\put(-8,8){$Z^{IV}:$}
\put(0,4){\line(1,0){32}}
\multiput(3,3)(4,0)7{$\bullet$}
\put(16,4){\oval(24,24)[t]}
\put(16,4){\oval(16,16)[t]}
\put(14,4){\oval(4,4)[t]}
\put(20,4){\line(0,1){15}}
\put(4,1){\makebox[0mm]{$\scriptstyle 7$}}
\put(8,1){\makebox[0mm]{$\scriptstyle 6$}}
\put(12,1){\makebox[0mm]{$\scriptstyle 5$}}
\put(16,1){\makebox[0mm]{$\scriptstyle 4$}}
\put(20,1){\makebox[0mm]{$\scriptstyle 3$}}
\put(24,1){\makebox[0mm]{$\scriptstyle 2$}}
\put(28,1){\makebox[0mm]{$\scriptstyle 1$}}
\end{picture}
\qquad\qquad
\beginpicture\setcoordinatesystem units <.4cm,.4cm>
\multiput{} at 0 0  4 4 /
\put{$\delta H(Y,Z^{IV}):$} at -3 2
\plot -.3 3.9  6.5 3.9 /
\plot -.3 4.1  6.5 4.1  /
\multiput{$\scriptstyle 0$} at 0 3.5  1 3.5  2 3.5  4 3.5  5 3.5      3 3.5
                                  .5 3  1.5 3  4.5 3  6.5 3    2.5 3  3.5 3
                                       1 2.5  3 2.5  6 2.5     2 2.5  4 2.5
                                           1.5 2  5.5 2        2.5 2
                                             2 1.5                    5 1.5
                                                 2.5 1
                                                    3.5 0 /
\multiput{$\scriptstyle 1$} at %
                                      3.5 2
                               3 1.5
                                      4.5 1
                                      4 .5 /
\multiput{$\cdot$} at        2.8 .7  3 .5  3.2 .3
                             3.3 1.2  3.5 1  3.7 .8
                             3.8 1.7  4 1.5  4.2 1.3
                             4.3 2.2  4.5 2  4.7 1.8
                             4.8 2.7  5 2.5  5.2 2.3
                             6.8 3.3  7 3.5  7.2 3.7 /
\setdots<2pt>
\plot 2.5 1.5  3.5 2.5   5 1  4 0  2.5 1.5 /
\endpicture
$$

Finally, an operation of type {\bf (B)} reduces the Hom-matrix to zero
and yields the module $Y$:  We replace the bipicket $B_2^{7,1}$ and the
picket $P_1^3$ in $Z^{IV}$ by the bipicket $B_2^{7,3}$ and the picket $P_1^1$
for $Y$.  We are done:

$$Y\leq_{\sf arc} Z^{IV}\leq_{\sf arc} Z''' \leq_{\sf arc}  Z'' \leq_{\sf arc}  Z' \leq_{\sf arc} Z $$

\section{The poset $({}_2\mathcal{S}_{\gamma}^{\beta},\leq_{\sf arc})$}\label{section-poset-arc}
%=================================================================

We discuss the maximal and minimal elements in the poset  
$({}_2\mathcal{S}_{\gamma}^{\beta},\leq_{\sf arc})$.  First, we give an example.

\subsection{An example}
%----------------------------------------------

\begin{ex}
We present the Hasse diagram of the poset $({}_2\mathcal{S}_{(3,2,1,1)}^{(4,3,3,2,1)},\leq_{\sf arc})$.
The numbers given on the right hand side are the dimensions of $\mathbb{V}_\Delta$ for diagrams $\Delta$
in the same row. These dimensions are computed using Theorem \ref{thm-second-main}.

$$\begin{picture}(110,75)(0,-10)
	\put(120,-5){
		\put(0,-2){\makebox{$160$}}
		\put(0,18){\makebox{$159$}}
		\put(0,33){\makebox{$158$}}		
		\put(0,48){\makebox{$157$}}
		\put(0,63){\makebox{$156$}}
	}
	\put(30,-5){
    \begin{picture}(20,8)(0,5)
      \put(0,4){\line(1,0){20}}
      \multiput(3,3)(4,0)4{$\bullet$}
      \put(6,4){\oval(4,4)[t]}
      \put(14,4){\oval(4,4)[t]}
		\put(8,4){\beginpicture
					\circulararc 360 degrees from 0 0 center at 0 -6.5
					\endpicture}
    \put(3,1){\makebox{$\scriptstyle 4$}}
    \put(7,1){\makebox{$\scriptstyle 3$}}
    \put(11,1){\makebox{$\scriptstyle 2$}}
    \put(15,1){\makebox{$\scriptstyle 1$}}
    
    \end{picture}
  }
  \put(60,-5){
  \begin{picture}(20,8)(0,5)
      \put(0,4){\line(1,0){20}}
      \multiput(3,3)(4,0)4{$\bullet$}
      \put(10,4){\oval(12,12)[t]}
      \put(10,4){\oval(4,4)[t]}
		\put(8,4){\beginpicture
					\circulararc 360 degrees from 0 0 center at 0 -6.5
					\endpicture}
    \put(3,1){\makebox{$\scriptstyle 4$}}
    \put(7,1){\makebox{$\scriptstyle 3$}}
    \put(11,1){\makebox{$\scriptstyle 2$}}
    \put(15,1){\makebox{$\scriptstyle 1$}}
    \end{picture}
  }
  \put(0,15){\begin{picture}(20,8)(0,5)
      \put(0,4){\line(1,0){20}}
      \multiput(3,3)(4,0)4{$\bullet$}
      \put(14,4){\oval(4,4)[t]}
      \put(4,4){\line(0,1){7}}
      \put(8,4){\line(0,1){7}}
		\put(8,4){\beginpicture
					\circulararc 360 degrees from 0 0 center at 0 -6.5
					\endpicture}
    \put(3,1){\makebox{$\scriptstyle 4$}}
    \put(7,1){\makebox{$\scriptstyle 3$}}
    \put(11,1){\makebox{$\scriptstyle 2$}}
    \put(15,1){\makebox{$\scriptstyle 1$}}
    
    \end{picture}
  }
  \put(30,15){
  \begin{picture}(20,8)(0,5)
      \put(0,4){\line(1,0){20}}
      \multiput(3,3)(4,0)4{$\bullet$}
      \put(10,4){\oval(4,4)[t]}
      \put(4,4){\line(0,1){7}}
      \put(16,4){\line(0,1){7}}
	\put(8,4){\beginpicture
					\circulararc 360 degrees from 0 0 center at 0 -6.5
					\endpicture}
    \put(3,1){\makebox{$\scriptstyle 4$}}
    \put(7,1){\makebox{$\scriptstyle 3$}}
    \put(11,1){\makebox{$\scriptstyle 2$}}
    \put(15,1){\makebox{$\scriptstyle 1$}}
    \end{picture}}
    \put(60,15){
    \begin{picture}(20,8)(0,5)
      \put(0,4){\line(1,0){20}}
      \multiput(3,3)(4,0)4{$\bullet$}
      \put(12,4){\oval(8,8)[t]}
      \put(8,4){\oval(8,8)[t]}
	\put(8,4){\beginpicture
					\circulararc 360 degrees from 0 0 center at 0 -6.5
					\endpicture}
    \put(3,1){\makebox{$\scriptstyle 4$}}
    \put(7,1){\makebox{$\scriptstyle 3$}}
    \put(11,1){\makebox{$\scriptstyle 2$}}
    \put(15,1){\makebox{$\scriptstyle 1$}}
    
    \end{picture}
    }
    \put(90,15){
    \begin{picture}(20,8)(0,5)
      \put(0,4){\line(1,0){20}}
      \multiput(3,3)(4,0)4{$\bullet$}
      \put(6,4){\oval(4,4)[t]}
      \put(12,4){\line(0,1){7}}
      \put(16,4){\line(0,1){7}}
		\put(8,4){\beginpicture
					\circulararc 360 degrees from 0 0 center at 0 -6.5
					\endpicture}
    \put(3,1){\makebox{$\scriptstyle 4$}}
    \put(7,1){\makebox{$\scriptstyle 3$}}
    \put(11,1){\makebox{$\scriptstyle 2$}}
    \put(15,1){\makebox{$\scriptstyle 1$}}
    \end{picture}
    }
    \put(15,30){
    \begin{picture}(20,8)(0,5)
      \put(0,4){\line(1,0){20}}
      \multiput(3,3)(4,0)4{$\bullet$}
      \put(12,4){\oval(8,8)[t]}
      \put(4,4){\line(0,1){7}}
      \put(12,4){\line(0,1){7}}
		\put(8,4){\beginpicture
					\circulararc 360 degrees from 0 0 center at 0 -6.5
					\endpicture}
    \put(3,1){\makebox{$\scriptstyle 4$}}
    \put(7,1){\makebox{$\scriptstyle 3$}}
    \put(11,1){\makebox{$\scriptstyle 2$}}
    \put(15,1){\makebox{$\scriptstyle 1$}}
    \end{picture}
    }
    \put(75,30){
    \begin{picture}(20,8)(0,5)
      \put(0,4){\line(1,0){20}}
      \multiput(3,3)(4,0)4{$\bullet$}
      \put(8,4){\oval(8,8)[t]}
      \put(8,4){\line(0,1){7}}
      \put(16,4){\line(0,1){7}}
		\put(8,4){\beginpicture
					\circulararc 360 degrees from 0 0 center at 0 -6.5
					\endpicture}
    \put(3,1){\makebox{$\scriptstyle 4$}}
    \put(7,1){\makebox{$\scriptstyle 3$}}
    \put(11,1){\makebox{$\scriptstyle 2$}}
    \put(15,1){\makebox{$\scriptstyle 1$}}
    \end{picture}
    }
    \put(45,45){
    \begin{picture}(20,8)(0,5)
      \put(0,4){\line(1,0){20}}
      \multiput(3,3)(4,0)4{$\bullet$}
      \put(10,4){\oval(12,12)[t]}
      \put(12,4){\line(0,1){7}}
      \put(8,4){\line(0,1){7}}
		\put(8,4){\beginpicture
					\circulararc 360 degrees from 0 0 center at 0 -6.5
					\endpicture}
    \put(3,1){\makebox{$\scriptstyle 4$}}
    \put(7,1){\makebox{$\scriptstyle 3$}}
    \put(11,1){\makebox{$\scriptstyle 2$}}
    \put(15,1){\makebox{$\scriptstyle 1$}} 
    \end{picture}
    }
    \put(45,60){
\begin{picture}(20,8)(0,5)
      \put(0,4){\line(1,0){20}}
      \multiput(3,3)(4,0)4{$\bullet$}
      \put(4,4){\line(0,1){7}}
		\put(8,4){\beginpicture
					\circulararc 360 degrees from 0 0 center at 0 -6.5
					\endpicture}
      \put(8,4){\line(0,1){7}}
      \put(12,4){\line(0,1){7}}
      \put(16,4){\line(0,1){7}}
    \put(3,1){\makebox{$\scriptstyle 4$}}
    \put(7,1){\makebox{$\scriptstyle 3$}}
    \put(11,1){\makebox{$\scriptstyle 2$}}
    \put(15,1){\makebox{$\scriptstyle 1$}}
    \end{picture}
    }
    {\color[rgb]{0.1,0.1,0.7}
  \put(25,25){\line(-1,-1){6}}
  \put(30,25){\line(1,-1){6}}
  \put(80,25){\line(-5,-1){30}}
  \put(90,25){\line(1,-1){6}}
  \put(52.5,40){\line(-3,-1){18}}
  \put(57.5,40){\line(3,-1){18}}}
  {\color[rgb]{0.1,0.7,0.1}
  \put(72.5,10){\line(0,-1){8}}
  \put(67.5,10){\line(-5,-2){24}}}
  {\color[rgb]{0.9,0.1,0.1}
  \put(42.5,10){\line(5,-2){24}}
  \put(12.5,10){\line(5,-2){24}}
  \put(97.5,10){\line(-5,-1){48}}
  \put(30,25){\line(5,-1){32}}
  \put(80,25){\line(-1,-1){6}}
  \put(53.5,55){\line(0,-1){4}}}
  \end{picture}$$
\end{ex}

\subsection{Maximal and minimal elements}
%----------------------------------------------------------------

\begin{prop}
 Let $\gamma,\beta$ be partitions. 
 \begin{enumerate}
  \item In the poset $({}_2\mathcal{S}_{\gamma}^{\beta},\leq_{\sf arc})$ there exists exactly one maximal element $M$.
  Its arc diagram has only poles and loops.
  \item If $\beta\setminus \gamma$ is vertical strip, then in the poset $({}_2\mathcal{S}_{\gamma}^{\beta},\leq_{\sf arc})$ there are exactly $c_{\alpha,\gamma}^\beta$ 
  minimal elements, where $c_{\alpha,\gamma}^\beta$ is the Littlewood-Richardson coefficient,
  $\alpha=(2,2,\ldots,2)$, if $|\beta|\setminus |\gamma|$ is even and $\alpha=(2,2,\ldots,2,1)$, 
  if $|\beta|\setminus |\gamma|$ is odd.
  Its arc diagram has no crossings, has only arcs, at most one pole, and may contain loops.
 \end{enumerate}
\label{prop-comb-prop}
\end{prop}

\begin{proof}
\begin{enumerate}
\item We start with the observation that any  object with an~arc can be modified by the~move {\bf (E)}.
In this way we obtain a~bigger object in our poset. It follows that each maximal element 
in the poset  $({}_2\mathcal{S}_{\gamma}^{\beta},\leq_{\sf arc})$ has no arc. 
The location of loops in the arc diagram is determined by the partitions $\beta$ and $\gamma$ so 
there exists exactly one maximal object, it has the required properties.
\item If the tableau is a~vertical strip, then in the arc diagram 
  there is no double pole at any given point.
Therefore, if there are at least two poles in an~arc diagram, we can choose two poles in the points $m$ and $r$ such that $m\neq r$.
It follows that, applying the move {\bf (E)}, we can obtain a~smaller object in our poset. 
We conclude that each minimal element in the poset  $({}_2\mathcal{S}_{\gamma}^{\beta},\leq_{\sf arc})$ 
has at most one pole. By \cite[Theorem 5.7]{kos-sch}, 
in the poset $({}_2\mathcal{S}_{\gamma}^{\beta},\leq_{\sf arc})$ there are exactly $c_{\alpha,\gamma}^\beta$ 
  minimal elements, where $c_{\alpha,\gamma}^\beta$ is the Littlewood-Richardson coefficient,
  $\alpha=(2,2,\ldots,2)$, if $|\beta|\setminus |\gamma|$ is even and $\alpha=(2,2,\ldots,2,1)$, 
  if $|\beta|\setminus |\gamma|$ is odd.
\end{enumerate}
\end{proof}

The following example shows that the assumption about the vertical strip
in Statement~2 of Proposition~\ref{prop-comb-prop} is necessary.

\begin{ex}
The poset ${}_{2}{\mathcal S}_{(2,2,1)}^{(3,3,2,1)}$ looks as follows.

 $$
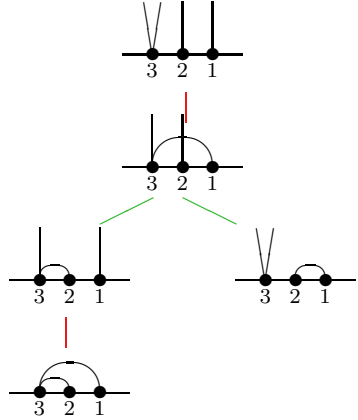
\begin{picture}(75,60)(0,-10)
  \put(30,0){
    	\begin{picture}(16,8)(0,3)
		\put(0,4){\line(1,0){16}}
		\multiput(3,3)(4,0)3{$\bullet$}
		\put(8,4){\oval(8,8)[t]}
      \put(6,4){\oval(4,4)[t]}
		\put(4,1){\makebox[0mm]{$\scriptstyle 3$}}
		\put(8,1){\makebox[0mm]{$\scriptstyle 2$}}
		\put(12,1){\makebox[0mm]{$\scriptstyle 1$}}
		\end{picture}
  }
  \put(30,15){ 
    	\begin{picture}(16,8)(0,3)
		\put(0,4){\line(1,0){16}}
		\multiput(3,3)(4,0)3{$\bullet$}
		\put(12,4){\line(0,1){7}}
		\put(4,4){\line(0,1){7}}
      \put(6,4){\oval(4,4)[t]}
		\put(4,1){\makebox[0mm]{$\scriptstyle 3$}}
		\put(8,1){\makebox[0mm]{$\scriptstyle 2$}}
		\put(12,1){\makebox[0mm]{$\scriptstyle 1$}}
		\end{picture}}
    \put(60,15){
		\begin{picture}(16,8)(0,3)
		\put(0,4){\line(1,0){16}}
		\multiput(3,3)(4,0)3{$\bullet$}
		\put(4,4){\line(1,6){1.15}}
		\put(4,4){\line(-1,6){1.15}}
		\put(10,4){\oval(4,4)[t]}
		\put(4,1){\makebox[0mm]{$\scriptstyle 3$}}
		\put(8,1){\makebox[0mm]{$\scriptstyle 2$}}
		\put(12,1){\makebox[0mm]{$\scriptstyle 1$}}
		\end{picture}
    }
    \put(45,30){	
		\begin{picture}(16,8)(0,3)
		\put(0,4){\line(1,0){16}}
		\multiput(3,3)(4,0)3{$\bullet$}
		\put(8,4){\line(0,1){7}}
		\put(4,4){\line(0,1){7}}
		\put(8,4){\oval(8,8)[t]}
		\put(4,1){\makebox[0mm]{$\scriptstyle 3$}}
		\put(8,1){\makebox[0mm]{$\scriptstyle 2$}}
		\put(12,1){\makebox[0mm]{$\scriptstyle 1$}}
		\end{picture}
    }    
	\put(45,45){
		\begin{picture}(16,8)(0,3)
		\put(0,4){\line(1,0){16}}
		\multiput(3,3)(4,0)3{$\bullet$}
		\put(12,4){\line(0,1){7}}
		\put(8,4){\line(0,1){7}}
		\put(4,4){\line(1,6){1.15}}
		\put(4,4){\line(-1,6){1.15}}
		\put(4,1){\makebox[0mm]{$\scriptstyle 3$}}
		\put(8,1){\makebox[0mm]{$\scriptstyle 2$}}
		\put(12,1){\makebox[0mm]{$\scriptstyle 1$}}
		\end{picture}
    }
  {\color[rgb]{0.1,0.7,0.1}
  \put(50.5,27){\line(-2,-1){7}}
  \put(54.5,27){\line(2,-1){7}}}
  {\color[rgb]{0.9,0.1,0.1}
  \put(37.5,11){\line(0,-1){4}}
  \put(53.5,41){\line(0,-1){4}}}
  \end{picture}$$
\end{ex}

\section{Two non-examples}\label{section-non-examples}
%-----------------------------------------

In the paper we fixed two partitions, $\beta$ and $\gamma$ and considered 
the variety $_2\mathbb V_\gamma^\beta$.  
We present two examples which show that similar results cannot be obtained
if we fix two other partitions, (that is, either $\alpha$ and $\gamma$ or $\alpha$ and $\beta$).

\subsection{Fixing partitions $\alpha$ and $\beta$}
%--------------------------------------------------

\begin{ex}
Let $\alpha,\beta$ be partitions such that $\alpha_1\leq 2$. 
We consider the full subcategory $_{\alpha}{\mathcal S}^\beta$ of $\mathcal S_2$ consisting of all objects $(N_\alpha,N_\beta,f)$. 
For $\alpha=(2)$ and $\beta=(m,r),$ where  $m>r+1$ we define two objects 
$X=B_2^{m,r}$ and $Y=P_0^{m}\oplus P_2^{r}$ in $_{\alpha}{\mathcal S}^\beta$. 
Using the table from Section~\ref{chapter-hom-tab} we compute that 
$$Y\leq_{\sf hom} X.$$
The object $X$ is indecomposable, so it is not true that:
$$Y\leq_{\sf ext} X.$$ 
\end{ex}

\subsection{Fixing partitions $\alpha$ and $\gamma$}
%---------------------------------------------------

\begin{ex}
Let $\alpha, \gamma$ be partitions such that $\alpha_1\leq 2$. 
We consider the full subcategory $_{\alpha}{\mathcal S}_\gamma$ of $\mathcal S_2$ consisting of all objects $(N_\alpha,N_\beta,f)$, 
such that $\Coker f\simeq N_\gamma$. For $\alpha=(2)$ and $\gamma=(m-1,r-1),$
where  $m>r+1$ we define two objects $X=B_2^{m,r}$ and $Y=P_2^{m+1}\oplus P_0^{r-1}$ in $_{\alpha}{\mathcal S}_\gamma$. 
Using the table from section \ref{chapter-hom-tab} we compute that 
$$Y\leq_{\sf hom} X.$$
The object $X$ is indecomposable, so it is not true that:
$$Y\leq_{\sf ext} X.$$
 
\end{ex}

\section{The dimensions of the orbits}\label{sec-dim-orbit}
%=====================================

In this section we present the proof of Theorem \ref{thm-second-main}.

\subsection{Previous results}
%----------------------------

We briefly recall notation and a result from \cite{kos-sch}. 

\smallskip
We consider the affine variety ${\mathbb H}_\alpha^\beta(k)=\allowbreak\Hom_k(N_\alpha(k),\allowbreak N_\beta(k))$
(consisting of all $|\beta|\times |\alpha|-$matrices with coefficients in $k$).
On ${\mathbb H}_\alpha^\beta(k)$ we consider the Zariski topology and on all subsets of
${\mathbb H}_\alpha^\beta(k)$ we work with the induced topology. Let ${\mathbb V}_{\alpha,\gamma}^\beta(k)$
be the subset of ${\mathbb H}_\alpha^\beta(k)$ consisting of all matrices that define a~monomorphism
$f:N_\alpha\to N_\beta$
in the category $\mathcal{N}(k)$ with $\Coker\,f\simeq N_\gamma$.  
On ${\mathbb V}_{\alpha,{\blue \gamma}}^\beta(k)$
acts the algebraic group $\Aut_{\mathcal{N}}(N_\alpha(k))\times \Aut_{\mathcal{N}}( N_\beta(k))$
via $(g,h)\cdot f=hfg^{-1}$. The orbits of this action correspond bijectively to isomorphism
classes of objects in $\mathcal S_{\alpha,\gamma}^\beta$.
For a map $f:N_\alpha(k)\to N_\beta(k)$, denote by $G_f$ the orbit of $f$ in
${\mathbb V}_{\alpha,{\blue \gamma}}^\beta(k)$.
% Let
% ${\mathbb V}_{\alpha,\gamma}^\beta(k)$ be the set of monomorphisms $f$ in
%  ${\mathbb V}_\alpha^\beta(k)$ such that
% $(N_\alpha(k),N_\beta(k),f)\in\mathcal S_{\alpha,\gamma}^\beta(k)$.

We recall a theorem which yields information about the dimensions of orbits 
in ${\mathbb V}_{\alpha,\gamma}^\beta(k)$. This theorem was proved in \cite[Section 5]{kos-sch}.

\begin{thm}\label{thm-second-main-old}
Let $k$ be an algebraically closed field, and let $\alpha,\beta,\gamma$ be partitions.
Suppose the arc diagram $\Delta$ of an invariant subspace
$Y=(N_\alpha,N_\beta,f)\in {\mathbb V}_{\alpha,\gamma}^\beta(k)$ has $x(\Delta)$ crossings. Then
$$\dim G_f\;=\; \deg h_{\alpha,\gamma}^\beta + \deg a_\alpha -x(\Delta),$$
where $\deg h_{\alpha,\gamma}^\beta = n(\beta)-n(\alpha)-n(\gamma)$
is the degree of the Hall polynomial $h_{\alpha,\gamma}^\beta(q)$ and
$\deg a_\alpha=|\alpha|+2 n(\alpha)$ is the degree of the polynomial
$a_\alpha(q)$ which counts the automorphisms of $N_\alpha(\mathbb F_q)$.\qed
\end{thm} 

\subsection{The dimensions of the orbits in ${}_{2}{\mathbb V}^\beta_\gamma.$}\label{subsec-dim-orbit}
%--------------------------------------------------------------------------
For a~group $G$ acting on a~variety $X$ and for $y\in X$, we denote by $G.y$ the stabilizer of $y$ in $G$.

\begin{proof}[Proof of Theorem~\ref{thm-second-main}]
Let $Y=(N_\alpha,N_\beta,f)\in {}_2\mathcal{S}_\gamma^\beta$ and let 
$F=(\varphi_\alpha,f,\varphi_\beta)\in {}_2{\mathbb V}_\gamma^\beta$ 
corresponds to $Y$. Denote by $\Delta$ the arc diagram diagram corresponding to $Y$ and assume 
that $\Delta$ has $x(\Delta)$ crossings.

It is well known that
\begin{equation}
\dim \mathcal{O}_F=\dim Gl(a,b) - \dim Gl(a,b).F\label{dim-O_f}
\end{equation}
where $a=|\alpha|$ and $b=|\beta|$. Moreover
\begin{equation}
  \dim \Aut(Y)=\dim Gl(a,b).F.\label{stab-aut}
\end{equation}

%From $\Delta$ we can read a partition $\alpha$ such that $Y=(N_\alpha,N_\beta, f)\in {\mathbb V}_{\alpha,\gamma}^\beta$.

Similarly,
\begin{eqnarray*}
\dim G_f & = & \dim \Aut_{\mathcal{N}}(N_\alpha(k))\times \Aut_{\mathcal{N}}( N_\beta(k)) \\
 & & \qquad - \dim (\Aut_{\mathcal{N}}(N_\alpha(k))\times \Aut_{\mathcal{N}}( N_\beta(k))).f.\label{dim-G_f}
\end{eqnarray*}

and
\begin{equation}
\dim \Aut(Y)=\dim (\Aut_{\mathcal{N}}(N_\alpha(k))\times \Aut_{\mathcal{N}}(N_\beta(k))).f.\label{stab-aut2}
\end{equation}

Substituting \ref{stab-aut} into \ref{dim-O_f} we obtain
\begin{equation}
\dim \mathcal{O}_F=\dim Gl(a,b) - \dim \Aut(Y)\label{dim-O_f2}
\end{equation}
and similarly for \ref{dim-G_f} we have
\begin{equation}
\dim G_f=\dim \Aut_{\mathcal{N}}(N_\alpha(k))\times \Aut_{\mathcal{N}}( N_\beta(k)) - \dim \Aut(Y).\label{dim-G_f2}
\end{equation}

Combining \ref{dim-O_f2} with \ref{dim-G_f2} yields 
\begin{equation}
\dim \mathcal{O}_F=\dim Gl(a,b) - (\dim \Aut_{\mathcal{N}}(N_\alpha(k))\times \Aut_{\mathcal{N}}( N_\beta(k))-\dim G_f).
\end{equation}

Finally, we obtain from Theorem \ref{thm-second-main-old}
$$\dim \mathcal{O}_F\;=\;|\beta|^2+|\alpha|^2-n(\alpha)-n(\beta)-n(\gamma)-|\beta|-x(\Delta).$$
\end{proof}

\subsection{The effect of a single arc move}
%-------------------------------------------

We finish by presenting an example which shows that there is no upper bound on the 
change of the orbit dimension resulting from a single arc move.

\begin{ex}  
Let $\beta=(3,1,1,\ldots,1), \gamma=(2)$. 
Then in the category ${\mathcal S}_\gamma^\beta$ 
we have only two objects $X=B_2^{3,1}\oplus\bigoplus_{i=1}^n P_1^1$
and $Y=P_1^{3}\oplus\bigoplus_{i=1}^{n+1} P_1^1$, up to isomorphism.
They have the following arc diagrams: 
$$
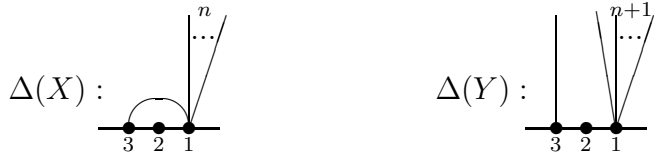
\begin{picture}(32,15)(0,3)
\put(-12,8){$\Delta(X):$}
\put(0,4){\line(1,0){16}}
\multiput(3,3)(4,0)3{$\bullet$}
\put(8,4){\oval(8,8)[t]}
\put(12.2,16){$...$}
\put(13.2,19){$\scriptstyle n$}
\put(12,4){\line(1,3){5}}
\put(12,4){\line(0,1){15}}
\put(4,1){\makebox[0mm]{$\scriptstyle 3$}}
\put(8,1){\makebox[0mm]{$\scriptstyle 2$}}
\put(12,1){\makebox[0mm]{$\scriptstyle 1$}}
\end{picture}
\qquad\qquad\qquad
\begin{picture}(32,15)(0,3)
\put(-12,8){$\Delta(Y):$}
\put(0,4){\line(1,0){16}}
\multiput(3,3)(4,0)3{$\bullet$}
\put(4,4){\line(0,1){15}}
\put(12.2,16){$...$}
\put(11.2,19){$\scriptstyle n+1$}
\put(12,4){\line(-1,6){2.6}}
\put(12,4){\line(1,3){5}}
\put(12,4){\line(0,1){15}}
\put(4,1){\makebox[0mm]{$\scriptstyle 3$}}
\put(8,1){\makebox[0mm]{$\scriptstyle 2$}}
\put(12,1){\makebox[0mm]{$\scriptstyle 1$}}\end{picture}
$$
% respectively.

Note that we can obtain the diagram $\Delta(Y)$ 
from the diagram $\Delta(X)$ by a single move {\bf(E)}. 
Using Theorem~\ref{thm-second-main}, we compute

\begin{eqnarray*}
\dim\mathbb{V}_{\Delta(X)} & = & |\beta|^2+|\alpha|^2-(2\cdot 0+1\cdot 1+1\cdot 2+\ldots+1\cdot n)\\
 & & \qquad -n(\beta)-n(\gamma)-|\beta|-0\\
 & = & |\beta|^2+|\alpha|^2-(1\cdot 0+1\cdot 1+1\cdot 2+\ldots+1\cdot n+1\cdot (n+1))\\
 & & \qquad -n(\beta)-n(\gamma)-|\beta|-0 \quad +(n+1)\\
 & = & \dim\mathbb{V}_{\Delta(Y)}\quad + (n+1).
\end{eqnarray*}

\end{ex}

  \bigskip
Address of the authors:

\parbox[t]{4cm}{\footnotesize\begin{center}
              Faculty of Mathematics\\
              and Computer Science\\
              Nicolaus Copernicus University\\
              ul.\ Chopina 12/18\\
              87-100 Toru\'n, Poland\end{center}}
\parbox[t]{4cm}{\footnotesize\begin{center}
              Faculty of Mathematics\\
              and Computer Science\\
              Nicolaus Copernicus University\\
              ul.\ Chopina 12/18\\
              87-100 Toru\'n, Poland\end{center}}
  \parbox[t]{4cm}{\footnotesize\begin{center}
              Department of\\
              Mathematical Sciences\\ 
              Florida Atlantic University\\
              777 Glades Road\\
              Boca Raton, Florida 33431\end{center}}

\parbox[t]{4cm}{\centerline{\footnotesize\tt kanies@mat.umk.pl}}
\parbox[t]{4cm}{\centerline{\footnotesize\tt justus@mat.umk.pl}}
           \parbox[t]{4cm}{\centerline{\footnotesize\tt markus@math.fau.edu}}

% Address of the authors:
%
% \parbox[t]{5.5cm}{\footnotesize\begin{center}
%               Faculty of Mathematics\\
%               and Computer Science\\
%               Nicolaus Copernicus University\\
%               ul.\ Chopina 12/18\\
%               87-100 Toru\'n, Poland\end{center}}
% \parbox[t]{5.5cm}{\footnotesize\begin{center}
%               Department of\\
%               Mathematical Sciences\\
%               Florida Atlantic University\\
%               777 Glades Road\\
%               Boca Raton, Florida 33431\end{center}}
%
% \smallskip \parbox[t]{5.5cm}{\centerline{\footnotesize\tt justus@mat.umk.pl}}
%            \parbox[t]{5.5cm}{\centerline{\footnotesize\tt markus@math.fau.edu}}


\begin{thebibliography}{999}
%===========================

% \bibitem{birkhoff} G.\ Birkhoff, {\it Subgroups of abelian groups,} Proc.\ Lond.\ Math.\ Soc., II.\
%   Ser.\ {\bf 38} (1934), 385--401.

\bibitem{bongartz} K.\ Bongartz, {\it On degenerations and extensions of finite dimensional
modules}, Adv.\ Math.\ {\bf 121} (1996), 245--287.

% \bibitem{bongartz1} K.\ Bongartz, {\it Degenerations for representations of tame
% quivers}, Ann.\ Sci.\ Ec.\ Norm.\ Super.\ {\bf 28} (1995), 647--668.


% \bibitem{JL2}  C.\ U.\ Jensen and H.\ Lenzing,
% {\it Model theoretic algebra with particular emphasis on fields, rings, modules},
% Algebra Logic Appl.\ 2. Gordon \& Breach, 1989.

%\bibitem{hump} J.\ E.\ Humphreys, {\it Linear Algebraic Groups}, Springer-Verlag, New York, Heidelberg, Berlin, 1975.

% \bibitem{kas} S.\ Kasjan, {\it Representation-directed algebras form an open scheme},
%  {\rm Colloq. Math.}  {\bf 93} (2002), 237--250.


%  \bibitem{kaskos}
%  S.\ Kasjan and J.\ Kosakowska, {\it On Lie algebras associated with
%  representation directed algebras}, J.\ Pure Appl.\ Algebra {\bf 214} (2010) 678--688.

% \bibitem{gk} T.\ Klein,
%   {\it The multiplication of Schur-functions and extensions of $p$-modules,}
%   J.\ Lond.\ Math.\ Soc.\ {\bf 43} (1968), 280--284.

\bibitem{klein} T.\ Klein,
  {\it The Hall polynomial,}
  J.\ Algebra {\bf 12} (1969), 61--78.

% \bibitem{kos03} J.\ Kosakowska, {\it Degenerations in a~class of matrix
% problems and prinjective modules}, J. Algebra {\bf 263} (2003), 262--277.

\bibitem{kos-sch} J.\ Kosakowska and M.\ Schmidmeier, {\it Operations on arc
diagrams and degenerations for invariant subspaces of linear
operators}, Trans. Amer. Math. Soc. {\bf 367} (2015), 5475-5505.

\bibitem{ks-survey} J.\ Kosakowska and M.\ Schmidmeier, 
  {\it Arc diagram varieties,}  Contemporary Mathematics series of the AMS
  {\bf 607} (2014), 205--224.


% \bibitem{kr} H.\ Kraft and Ch.\ Riedtmann, {\it Geometry of representations of quivers},
% In: Representations of algebras, London Math.\ Soc.\ Lecture Note Ser.\ {\bf 116,} 1986.

% \bibitem{lw} S.\ Lang and A.\ Weil, {\it Numbers of points of varieties in finite fields},
% Amer.\ J.\ Math.\ {\bf 76} (1954), 819--827.

% \bibitem{macd} I.\ G.\ Macdonald, {\it Symmetric Functions and Hall Polynomials},
% Oxford University Press, 1995.

\bibitem{riedtmann} C.\ Riedtmann, {\it Degenerations for representations of quiver with
relations}, Ann.\ Sci.\ Ec.\ Norm.\ Super.\ {\bf 4} (1986), 275--301.


% \bibitem{rin1} C.\ M.\ Ringel, {\it Hall algebras}, Banach Center
%  Publ., Vol.\ {\bf 26,} Warsaw 1990, 433--447.

% {\red \bibitem{rs} C.\ M.\ Ringel and M.\ Schmidmeier,
%   {\it Invariant subspaces of nilpotent linear operators. I},
%   J.\ Reine Angew.\ Math.\ {\bf 614} (2008), 1--52.}

\bibitem{s-b} M.\ Schmidmeier,
  {\it Bounded submodules of modules,}
  J.\ Pure Appl.\ Algebra {\bf 203} (2005), 45--82.

\bibitem{sch} M.\ Schmidmeier,
  {\it Hall polynomials via automorphisms of short exact sequences,}
  Algebr.\ Represent.\ Theory {\bf(15)} (2012), 449--481.

% \bibitem{simson} D.\ Simson,
%   {\it Representation types of the category of subprojective representations
%     of a finite poset over $K[t]/(t^m)$ and a solution of a Birkhoff type problem,}
%   J.\ Algebra {\bf 311} (2007), 1--30.

% \bibitem{zwara} G.\ Zwara, {\it Degenerations for representations of extended Dynkin
% quivers}, Comment.\ Math.\ Helv.\ {\bf 73} (1998), 71--88.


\end{thebibliography}
\end{document}